\documentclass[oneside]{amsart}
\usepackage{enumerate,amssymb}
\newtheorem{thm}{Theorem}[section]
\newtheorem{coro}[thm]{Corollary}
\newtheorem{prop}[thm]{Proposition}
\newtheorem{lem}[thm]{Lemma}
\theoremstyle{definition}
\newtheorem{defn}[thm]{Definition}
\newcommand{\Rset}{\mathbb{R}}
\newcommand{\Nset}{\omega}
\newcommand{\Cset}{2^\Nset}
\newcommand{\CCset}{2^{<\Nset}}
\newcommand{\Pset}{\Nset^\Nset}
\newcommand{\UPset}{\Nset^{\uparrow\Nset}}

\newcommand{\eps}{\varepsilon}
\newcommand{\del}{\delta}
\newcommand{\subs}{\subseteq}
\newcommand{\sups}{\supseteq}
\newcommand{\clos}[1]{\overline{#1}}
\newcommand{\upto}{{\nearrow}}

\newcommand{\concat}{^\smallfrown\hspace{-0.3ex}}
\newcommand{\rest}{{\restriction}}
\renewcommand{\leq}{\leqslant}
\renewcommand{\geq}{\geqslant}

\newcommand{\abs}[1]{\lvert#1\rvert}
\newcommand{\seq}[1]{\langle#1\rangle}
\newcommand{\seqeps}{\seq{\eps_n}}
\newcommand{\Seqeps}{\seqeps\in(0,\infty)^\Nset}
\newcommand{\si}{$\sigma$\nobreakdash-}
\newcommand{\hm}{\mathcal H}
\newcommand{\uhm}{\overline{\mathcal H}{}}
\newcommand{\gbox}{\boldsymbol{\nu}}
\newcommand{\ubox}{\gbox}
\newcommand{\lbox}{\underline{\gbox}}
\newcommand{\dbox}{\underrightarrow{\gbox}\mkern-3.2mu}

\DeclareMathOperator{\id}{id}

\DeclareMathOperator{\hdim}{\dim_{\mathsf{H}}}
\DeclareMathOperator{\uhdim}{\overline{\dim}_{\mathsf{H}}}

\DeclareMathOperator{\lbdim}{\underline{\dim}_{\mathsf{B}}}
\DeclareMathOperator{\ubdim}{\overline{\dim}_{\mathsf{B}}}

\DeclareMathOperator{\dpdim}{\underrightarrow{\dim}_{\mathsf{P}}}
\DeclareMathOperator{\pdim}{\dim_{\mathsf{P}}}
\DeclareMathOperator{\updim}{\overline{\dim}_{\mathsf{P}}}

\DeclareMathOperator{\add}{\mathsf{add}}

\DeclareMathOperator{\dom}{dom}
\DeclareMathOperator{\diam}{diam}

\newcommand{\NN}{\mathcal N}
\newcommand{\NNs}{\mathcal N_\sigma}
\newcommand{\MM}{\mathcal M}
\newcommand{\EE}{\mathcal E}

\newcommand{\CC}{\mathcal C}
\newcommand{\HH}{\mathbb H}

\newcommand{\smz}{$\mathcal{S}\mkern-1mu\mathfrak{mz}$}

\newcommand{\mc}[1]{\mathcal{#1}}

\newcommand{\CCC}{\boldsymbol{\mc C}}
\newcommand{\upnull}{$\overline{\boldsymbol{\mc P}}$-{\rm null}}
\renewcommand{\upnull}{$\boldsymbol{\mc P}$-{\rm null}}

\newcommand{\dpnull}{$\underrightarrow{\boldsymbol{\mc P}}$-{\rm null}}
\newcommand{\Dpnull}{$\boldsymbol{\underset{\rightarrow}{\mc P}}$-{\rm null}}
\newcommand{\uhnull}{$\overline{\boldsymbol{\mc H}}$-{\rm null}}
\newcommand{\hnull}{$\boldsymbol{\mc H}$-{\rm null}}

\newcommand{\Tprime}{\ensuremath{\mathrm{(T')}}}

\newcommand{\fmany}{\forall^\infty}

\newenvironment{enum}{\begin{enumerate}[\rm(i)]}{\end{enumerate}}
\newenvironment{itemyze}%
{\begin{list}{\textbullet}{\setlength{\labelwidth}{1ex}\setlength{\leftmargin}{1.1em}}}%
{\end{list}}

\newcommand{\Implies}{\ensuremath{\Rightarrow}}
\begin{document}
\title
[Small sets of reals through the prism of fractal dimensions]
{Small sets of reals\\ through the prism of fractal dimensions}
\author{Ond\v rej Zindulka}
\address
{Department of Mathematics\\
Faculty of Civil Engineering\\
Czech Technical University\\
Th\'akurova 7\\
160 00 Prague 6\\
Czech Republic}
\email{zindulka@mat.fsv.cvut.cz}
\urladdr{http://mat.fsv.cvut.cz/zindulka}
\thanks{Work on this project was conducted during the author's sabbatical stay
at the Instituto de matem\'aticas, Unidad Morelia,
Universidad Nacional Auton\'oma de M\'exico supported by CONACyT gratnt no.~125108.
The project was also supported by
The Czech Republic Ministry of Education, Youth and Sport,
research project BA MSM 210000010}
\begin{abstract}
A separable metric space $X$ is an \hnull{} set if any uniformly continuous
image of $X$ has Hausdorff dimension zero. \uhnull{}, \dpnull{} and \upnull{} sets
are defined likewise, with other fractal dimensions in place of Hausdorff dimension.
We investigate these sets and show that in $\Cset$ they coincide,
respectively, with strongly null, meager-additive, \Tprime{} and
null-additive sets. Some consequences:
A subset of $\Cset$ is meager-additive if and only if it is
$\EE$-additive; if $f:\Cset\to\Cset$ is continuous and $X$ is meager-additive,
then so is $f(X)$, and likewise for null-additive and \Tprime{} sets.
\end{abstract}
\maketitle
\section{Introduction}
\label{sec:intro}

\subsection*{Strong measure zero and Hausdorff dimension}

By the definition due to Borel, a metric space $X$ has
\emph{strong measure zero} (\smz) if for any sequence $\seq{\eps_n}$
of positive numbers there is a cover $\{U_n\}$ of $X$
such that $\diam U_n\leq\eps_n$ for all $n$.
By the famous Galvin--Mycielski--Solovay Theorem~\cite{GMS},
a subset $X$ of the line has \smz{} if and
only if there is no meager set $M$ such that $X+M=\Rset$.
The same theorem holds for subsets of the Cantor set $\Cset$,
as proved e.g.~in\cite[1.14]{MR2177439}.

It is almost obvious that a \smz{} space has Hausdorff
dimension zero. Since \smz{} is preserved by uniformly continuous
mappings, it follows that any uniformly continuous image of a \smz{}
space has Hausdorff dimension zero. It is not difficult to prove that
the latter property actually characterizes \smz{}.
To be more precise, denote $\hdim$ Hausdorff dimension and
say that a metric space $X$ is \hnull{} if $\hdim f(X)=0$ for each uniformly
continuous mapping of $X$ into another metric space.
Then a metric space is \smz{} if and only if it is \hnull{}, and
thus Galvin--Mycielski--Solovay Theorem for $\Cset$ can be phrased
``$X\subs\Cset$ is \hnull{} if and only if there is no meager set such that
$X+M=\Cset$.''
The essence of this theorem can be traced back to Besicovitch
papers~\cite{MR1555386,MR1555389}.

In summary, we thus have three essentially different descriptions of \smz{} sets
in $\Cset$: ``combinatorial'' (Borel's definition), ``fractal'' (by Hausdorff dimension)
and ``algebraic'' (there is no meager set such that $X+M=\Cset$).

\subsection*{Small spaces from other fractal dimensions}

One may, just for curiosity, investigate spaces that are defined by the same pattern
as \hnull{} spaces, replacing Hausdorff dimension with some other fractal dimension.
For instance, for packing dimension $\pdim$: Call $X$ to be \upnull{} if
$\pdim f(X)=0$ for each uniformly continuous mapping of $X$ into another metric space.
Besides \hnull{} and \upnull{} spaces we consider also \uhnull{} and \dpnull{}
spaces arising from the so called upper Hausdorff dimension and directed
lower packing dimension, respectively. The detailed exposition
of all four dimensions and the fractal measures behind them is provided below.

Let us point out that all of these small sets are consistently countable:
Recall the \emph{Borel Conjecture}. It is the statement ``\emph{Every \smz{} set
is countable}''.
As proved by Laver~\cite{MR0422027}, Borel Conjecture is consistent. As proved in
Theorem~\ref{hnullsmz}, every \hnull{} space is \smz{}. Thus it is
consistent that \hnull{} sets, and \emph{a fortiori} \uhnull{}, \dpnull{} and
\upnull{} sets are countable. On the other hand, under the Continuum Hypothesis
there are uncountable \upnull{} sets.

\subsection*{Meager-additive sets and the like}

Let $\Cset$ denote the usual Cantor cube. The coordinatewise addition
makes $\Cset$ a compact topological group. Denote its Haar measure $\mu$;
it is the usual product measure. Provide $\Cset$ with the usual least
difference metric.

There are three common \si ideals on $\Cset$: $\MM$, the ideal of
meager sets; $\NN$, the ideal of $\mu$-null sets; and $\EE$, the ideal generated by
$\mu$-null $F_\sigma$-sets.

Given two sets $A,B\subs\Cset$, their sum is defined by $A+B=\{a+b:a\in A,b\in B\}$.
Recall that, given an ideal $\mc J$ on $\Cset$, a set $X\subs\Cset$
is termed \emph{$\mc J$-additive} if $X+J\in\mc J$ for all $J\in \mc J$.
Thus we have notions of \emph{$\NN$-additive}, \emph{$\MM$-additive} and
\emph{$\EE$-additive} sets.
Say that $X\subs\Cset$ is \emph{strongly null} if there is no set $M\in\MM$
such that $X+M=\Cset$.

These notions (except perhaps $\EE$-additive) received a lot of attention.
Shelah~\cite{MR1324470} provided several
combinatorial characterizations of $\NN$-additive and $\MM$-additive sets
and proved that every
$\NN$-additive set is $\MM$-additive, see also~\cite{MR1350295}.
Yet another related notion, \Tprime{}-sets, was introduced and investigated by
Nowik and Weiss~\cite{MR1905154}. They proved, in particular, the implications
$\NN$-additive$\Rightarrow\Tprime\Rightarrow\MM$-additive.

\subsection*{The match}
The goal of the present paper is to prove that the five notions of the last
paragraph match the notions based on fractal dimensions introduced in the
next to last paragraph in a perhaps unexpected manner:
\begin{thm}\label{bigthm}
For subsets of $\Cset$, the following diagram holds.
\end{thm}
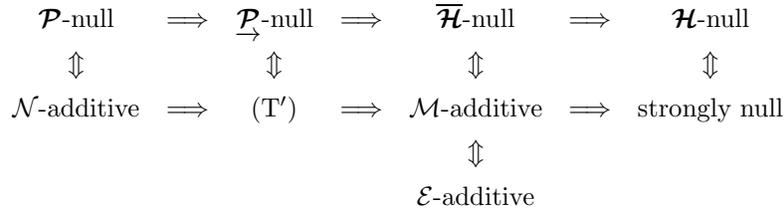
\begin{center}\renewcommand{\arraystretch}{1.4}
\newcommand{\lra}{\Longrightarrow}
\newcommand{\uda}{\Updownarrow}
\begin{tabular}{ccccccc}
\upnull        & $\lra$ & \dpnull & $\lra$ & \uhnull        & $\lra$ & \hnull\\
$\uda$         &        & $\uda$  &        & $\uda$         &        &$\uda$\\
$\NN$-additive & $\lra$ & \Tprime & $\lra$ & $\MM$-additive & $\lra$ & strongly null\\
               &        &         &        &$\uda$\\
               &        &         &        &$\EE$-additive
\end{tabular}
\end{center}
%
The upper line implications follow trivially from definitions and the
chain~\eqref{basicinequality} of inequalities between the respective fractal
dimensions. Thus once the vertical equivalences are proved, the diagram is settled.
The equivalences are subject to theorems~\ref{mainSN},
\ref{mainME}, \ref{mainNN} and \ref{mainTT} proven below.


\smallskip
The paper is organized in eleven sections. First ten sections form three parts.
The preliminary part consists of sections~\ref{sec:intro}--\ref{sec:measures}.
In section~\ref{sec:null} the four notions of smallness based on fractal dimensions
are introduced and
section~\ref{sec:measures} reviews the fractal measures behind the four dimensions.
In the second part consisting of sections~\ref{sec:hnull}--\ref{sec:dpnull}
elementary properties of the four types of smallness are established within the
framework of separable metric spaces.
In the third part consisting of sections~\ref{sec:meager}--\ref{sec:TT}
we investigate further properties of the
four kinds of small sets within the Cantor set $\Cset$ and in particular we prove
the vertical equivalences from the above diagram. In the concluding section
we provide some comments and list several open problems.

Some common notation used throughout the paper includes $\abs{A}$ for
the cardinality of a set $A$, $\Nset$ for the set of natural numbers,
$[\Nset]^\Nset$ for the collection of infinite
subsets of $\Nset$,
and $\UPset$ for the family of nondecreasing
unbounded sequences of natural numbers.

\section{Sets of small fractal dimension}
\label{sec:null}
We first briefly describe the four kinds of fractal dimensions under
consideration. More details and references are provided  in the next section.
Let $X$ be a metric space.

\subsection*{Hausdorff dimensions}
Hausdorff dimension is well-known. We shall denote it $\hdim X$.
The following modification of Hausdorff dimension, called the
\emph{upper Hausdorff dimension} can be derived from the Hausdorff dimension
as follows:
Let $\displaystyle X^\star$ denote the completion of $X$ and define
$$
  \uhdim X=\inf\{\hdim K:\text{$K$ is \si compact, }X\subs K\subs X^\star\}.
$$
\subsection*{Packing dimensions}
The \emph{covering number function} $N_X(\del)$ of a nonempty metric space $X$ is defined
as the minimal number of sets of diameter at most $\del$ needed to cover $X$.
The \emph{upper} and \emph{lower box dimensions}
of $X$ are defined, respectively, by
$\ubdim X =\varlimsup_{\del\to 0}\frac{\log N_X(\del)}{\abs{\log \del}}$
and
$\lbdim X =\varliminf_{\del\to 0}\frac{\log N_X(\del)}{\abs{\log \del}}$.
The \emph{upper packing dimension} of $X$ is defined by
$$
  \updim X=
  \inf\bigl\{\sup_n\ubdim X_n:\{X_n\}\text{ is a cover of $X$}\bigr\}.
$$
The following dimension, akin to the so called lower packing dimension,
occurs naturally in the investigation
of cartesian products of fractal sets, see~\cite{ZinPack}.
Write $X_n\upto X$ to denote that $\{X_n\}$ is an increasing
sequence of sets with union $X$.
$$
  \dpdim X=\inf\bigl\{\sup_n\lbdim X_n:X_n\upto X\bigr\}.
$$
The following chain of inequalities holds for any space $X$,
see~\eqref{basicIneq}
\begin{equation}\label{basicinequality}
  \hdim X\leq\uhdim X\leq\dpdim X\leq\updim X
\end{equation}
with examples showing that each of the inequalities may be strict.

\subsection*{Small sets from fractal dimensions}
Using a common pattern we define four notions of small sets
arising from the four fractal dimensions. Say that $f$ is a mapping on
a metric space $X$ if $f:X\to Y$, where $Y$ is a metric space.

\begin{defn}
Let $X$ be a separable metric space. Define $X$ to be
\begin{itemyze}
\item
\upnull{} if $\updim f(X)=0$ for each uniformly continuous mapping $f$ on $X$,
\item
\dpnull{} if $\dpdim f(X)=0$ for each uniformly continuous mapping $f$ on $X$,
\item
\uhnull{} if $\uhdim f(X)=0$ for each uniformly continuous mapping $f$ on $X$,
\item
\hnull{} if $\hdim f(X)=0$ for each uniformly continuous mapping $f$ on $X$.
\end{itemyze}
\end{defn}
The inequalities \eqref{basicinequality} yield the upper line of the Theorem~\ref{bigthm} diagram:
\begin{equation}\label{basicinequality2}
  \text{\upnull}\implies
  \text{\dpnull}\implies
  \text{\uhnull}\implies
  \text{\hnull}
\end{equation}
It is straightforward from the definitions that all of the four properties are preserved by
uniformly continuous mappings:
\begin{prop}
A uniformly continuous image of a \upnull{} set is \upnull{}.
Analogous statements hold for \dpnull{}, \uhnull{}
and \hnull{} sets.
\end{prop}
Each of the four notions is \si additive, i.e.~%
for any $X$ the \upnull{} subsets of $X$ form a \si additive ideal
and likewise for \dpnull{}, \uhnull{} and \hnull.
This is an obvious consequence
of the countable stability of the corresponding dimensions, except for
\dpnull{}, for which it is nontrivial and will be proved in Corollary~\ref{dpnullideal}.

\section{Review of fractal measures}
\label{sec:measures}
Before getting any further we have to review fractal measures
that are behind the four fractal dimensions.

Let $X$ be a space and $d$
its metric. If $A\subs X$, then $dA$ denotes the diameter of $A$
and if $\mc A$ is a family of subsets of $X$, then $d\mc A=\sup_{A\in\mc A} dA$.
A closed ball of radius $r$ centered in $x$ is denoted $B(x,r)$.

Let $\HH$ denote the set of all functions $h:[0,\infty)\to[0,\infty)$
that are nondecreasing, right-continuous,
and satisfy $h(r)=0$ iff $r=0$.
Elements of $\HH$ are called \emph{Hausdorff functions}.
The following is the common ordering of $\HH$:
$$
  g\prec h\quad \overset{\mathrm{def}}{\equiv}
  \quad\lim_{r\to0+}\frac{h(r)}{g(r)}=0.
$$
Given $s>0$, we shall write $h\prec s$ to abbreviate that $h\prec g_s$, where
$g_s(r)=r^s$.
Notice that for any sequence of $h_n\in\HH$ there is
$h\in\HH$ such that $h\prec h_n$ for all $n$.

Let $\tau$ be a pre-measure, i.e.~a monotone $[0,\infty]$-valued set function
such that $\tau(\emptyset)=0$.
We shall denote $\NN(\tau)=\{A\in\dom\tau:\tau(A)=0\}$ the family of
negligible sets, and in case $\tau$ is not \si subadditive, $\NNs(\tau)$
denotes the \si ideal generated by $\NN(\tau)$.

The following operation
known as Munroe's \emph{Method I construction} assigns to any pre-measure $\tau$
the maximal \si additive measure majorized by $\tau$:
$$
  \tau^\mathrm{I}(E)=\inf\left\{\sum_{n\in\Nset}\tau(E_n):
  E\subs\bigcup_{n\in\Nset}E_n\right\}.
$$
\subsection*{Hausdorff measure}
If $\del>0$, a cover $\mc A$ of a set $E\subs X$ is termed a \emph{$\del$-cover} if $d\mc A\leq\del$.
Fix $h\in\HH$. The \emph{$h$-dimensional Hausdorff measure} $\hm^h(E)$ of
a set $E$ in a space $X$ is defined thus:
For each $\delta>0$ set
$$
  \hm^h_\delta(E)=
  \inf\left\{\sum_{n\in\Nset}h(dE_n):
  \text{$\{E_n\}$ is a countable $\delta$-cover of $E$}\right\}
$$
and put
$
  \hm^h(E)=\sup_{\delta>0}\hm^h_\delta(E).
$
Properties of Hausdorff measures are well-known. We point out that $\hm^h$
(or rather its restriction to Borel sets)
is a $G_\del$-regular Borel measure and the following facts.
Recall that a countable cover of a set $E$ is termed a $\lambda$-cover
if every point of $E$ is contained in infinitely many $U_n$'s.
Reference:~\cite{MR0281862}.
\begin{prop}\label{lambda}
$\hm^h(E)=0$ if and only if $E$ admits a countable
$\lambda$-cover $\{E_n\}$ such that $\sum_{n\in\Nset}h(dE_n)<\infty$.
\end{prop}
\begin{lem}\label{lemHaus}
\begin{enum}
\item If $\hm^h(X)<\infty$ and $h\prec g$, then $\hm^g(X)=0$.
\item If $\hm^h(X)=0$, then there is $g\prec h$ such that $\hm^g(X)=0$.
\end{enum}
\end{lem}
%
\subsection*{Upper Hausdorff measure}
The following variation of $\hm^h$ plays an important role
in our considerations.
It is defined thus:
For each $\delta>0$ set
$$
  \uhm^h_\delta(E)=
  \inf\left\{\sum_{n=0}^N h(dE_n):
  \text{$\{E_n:n\leq N\}$ is a finite $\delta$-cover of $E$}\right\}
$$
and put
$
  \uhm_0^h(E)=\sup_{\delta>0}\uhm^h_\delta(E),
$
so that the only difference from $\hm^h$ is that only finite covers are taken in account.
This also makes $\uhm_0^h$ finitely additive, but not \si additive.
Put $\uhm^h(E)=(\uhm_0^h)^\mathrm{I}(E)$.

We list some properties of $\uhm_0^h$ and $\uhm^h$.
Some of them will be utilized below and some are provided just to shed more light on the
notion of upper Hausdorff measure.
The routine proofs are omitted.
\begin{lem}\label{lem1}
\begin{enum}
\item If $\uhm_0^h(E)<\infty$, then $E$ is totally bounded.
\item $\uhm_0^h(E)=\uhm_0^h(\clos E)$.
\item $\uhm_0^h(E)=\hm^h(E)$ if $E$ is compact.
\item If $X$ is complete, $E\subs X$ and $E\in\NNs(\uhm_0^h)$, then there is a \si compact set
$K\sups E$ such that $\hm^h(K)=0$.
\item If $X$ is complete and $E\subs X$, then
$\uhm^h(E)=\inf\{\hm^h(K):\text{$K\sups E$ is \si compact}\}$.
\item In particular $\uhm^h(E)=\hm^h(E)$ if $E$ is \si compact.
\item If $g\prec h$ and $\uhm^g(E)<\infty$, then $E\in\NNs(\uhm_0^h)$.
\end{enum}
\end{lem}
We shall need the counterpart of~\ref{lambda} for $\uhm^h$ at several occasions.
\begin{defn}
Let $\seq{U_n}$ be a cover of a set $X$.
Recall that $\seq{U_n}$ is called
a $\gamma$-cover if each $x\in X$ belongs to all but finitely
many $U_n$.

Recall that a cover $\seq{U_n}$ of a set $X$ is called
\emph{$\gamma$-groupable}
if there is a partition $\Nset=I_0\cup I_1\cup I_2\cup\dots$
into finite sets (or intervals, which makes no difference) such that
the sequence $\seq{\bigcup_{n\in I_j}U_n:j\in\Nset}$ is a $\gamma$-cover.
The finite families $\{U_n:n\in I_j\}$
will be occasionally termed \emph{witnessing groups}.
\end{defn}

\begin{lem}\label{gammagr}
$E\in\NNs(\uhm_0^h)$ if and only if $E$ has a $\gamma$-groupable
cover $\seq{U_n}$ such that $\sum_{n\in\Nset}h(dU_n)<\infty$.
\end{lem}
\begin{proof}
\Implies: Let $E_n\upto E$, $\uhm^g_0(E_n)=0$.
For each $n$ let $\mc G_n$ be a finite cover of $E_n$ such that
$\sum_{G\in\mc G_n}g(dG)<2^{-n}$. Put $\mc G=\bigcup_n\mc G_n$.
The witnessing groups are $\mc G_n$.

$\Leftarrow$: Let $\mc G_j$ be the witnessing groups.
Put $E_k=\bigcap_{j\geq k}\bigcup\mc G_j$. Fix $k$. The set $E_k$ is covered
by each $\mc G_j$, $j\geq k$, and $\sum_{G\in\mc G_j}g(dG)$ is as
small as needed if $j$ is large enough. Hence $\uhm^g_0(E_k)=0$.
\end{proof}

\subsection*{Box measures}
We could develop the theory of \upnull{} and \dpnull{} spaces from
packing measures. But since they are rather unpleasant to work with,
we make use of the following variations instead. They are directly related to the
above definitions of packing dimensions and are easier to work with. Given $h\in\HH$, set
$$
  \ubox_0^h(E)=\limsup_{r\to0}N_E(r)\cdot h(r).
$$
The \emph{$h$-dimensional box measure} of
$E\subs X$ is defined by $\ubox^h(E)=(\ubox_0^h)^\mathrm{I}(E)$.
\begin{lem}\label{lemP}
\begin{enum}
\item If $\ubox_0^h(E)<\infty$, then $E$ is totally bounded.
\item $\ubox_0^h(E)=\ubox_0^h(\clos E)$.
\item If $X$ is complete, $E\subs X$ and $E\in\NNs(\ubox_0^h)$, then there is a \si compact set
$K\sups E$ such that $\ubox^h(K)=0$.
\item If $g\prec h$ and $\ubox^g(E)<\infty$, then $E\in\NNs(\ubox_0^h)$.
\item If $E\in\NNs(\ubox_0^h)$, then there is $g\prec h$ such that $E\in\NNs(\ubox_0^g)$.
\end{enum}
\end{lem}
The set functions
$\hm^h,\uhm^h$ and $\ubox^h$ are are Borel outer measures, i.e.~Borel-regular
outer measures whose restrictions
to the algebra of Borel sets are Borel measures.
The set function that we shall introduce now is not
really a measure, as it is defined from the lower box contents
that is not finitely subadditive. Let
$$
  \lbox_0^h(E)=\liminf_{r\to0}N_E(r)\cdot h(r).
$$
Write $E_n\upto E$ to denote that $\{E_n\}$ is an increasing
sequence of sets with union $E$. Define
$$
  \dbox^h(E)=\inf\left\{\sup_{n\in\Nset}\lbox_0^h(E_n):E_n\upto E\right\}
$$
This is a ``directed'' variation of Method I construction.
\begin{lem}\label{lemDP}
\begin{enum}
\item If $\lbox_0^h(E)<\infty$, then $E$ is totally bounded.
\item $\lbox_0^h(E)=\lbox_0^h(\clos E)$.
\item If $g\prec h$ and $\dbox^g(E)<\infty$, then $E\in\NNs(\dbox_0^h)$.
\item If $E\in\NNs(\lbox_0^h)$, then there is $g\prec h$ such that $E\in\NNs(\lbox_0^g)$.
\end{enum}
\end{lem}

In the common case when $h(r)=r^s$ for some
$s>0$ we write $\hm^s$ for $\hm^h$, and the same license is used for all
pre-measures and measures under consideration.

It is easy to check that
\begin{equation}\label{basicIneq}
  \hm^g\leq\uhm^g\leq\dbox^g\leq\ubox^g
\end{equation}
and that the three measures $\uhm^g$, $\dbox^g$ and $\ubox^g$ satisfy the following
continuity property.
\begin{lem} \label{IncreasingSetsLemma}
If $\uhm^g(X)<s$, then there is a sequence $X_n\upto X$ such that
$\sup\uhm_0^g(X_n)<s$.
Analogous statements hold for $\ubox^g$ and $\dbox^g$.
\end{lem}

\subsection*{Cartesian products}
Given two metric spaces $X_1$ and $X_2$ with respective metrics
$d_1$ and $d_2$, provide the cartesian product $X_1\times X_2$
with the maximum metric
$$
  d\bigl((x_1,x_2),(y_1,y_2)\bigr)=\max(d_1(x_1,y_1),d_2(x_2,y_2)).
$$
A Hausdorff function $h$ is of \emph{finite order}
(or \emph{blanketed} or satisfies \emph{doubling condition}) if
$\limsup_{r\to0}\frac{h(2r)}{h(r)}<\infty$.
\begin{lem}\label{howroyd}
Let $X,Y$ be metric spaces and $h,g$ Hausdorff functions. Then
\begin{enum}
\item $\ubox^{hg}(X\times Y)\leq\ubox^h(X)\,\ubox^g(Y)$,
\item $\hm^{hg}(X\times Y)\leq\hm^h(X)\,\ubox^g(Y)$,
\end{enum}
provided the rightmost products are not $0\cdot\infty$ or $\infty\cdot0$.

If $h,g$ are of finite order, then
\begin{enum}\setcounter{enumi}{2}
\item $\hm^h(X)\,\hm^g(Y)\leq\hm^{hg}(X\times Y)$,
\item $\uhm^h(X)\,\uhm^g(Y)\leq\uhm^{hg}(X\times Y)$,
\end{enum}
and there is a constant $c>0$ depending only on $g$ and $h$ such that
\begin{enum}\setcounter{enumi}{4}
\item $\ubox^h(X)\,\dbox^g(Y)\leq c\,\ubox^{hg}(X\times Y)$,
\item $\dbox^h(X)\,\dbox^g(Y)\leq c\,\dbox^{hg}(X\times Y)$.
\end{enum}
\end{lem}
\begin{proof}[Proof in outline]
(i) easily follows from the definitions and the obvious inequality
$N_{X\times Y}(\del)\leq N_X(\del)N_Y(\del)$.

(ii) It is clearly enough to prove that $\hm^{hg}(X\times Y)\leq\hm^h(X)\,\ubox_0^g(Y)$.
Fix $\eps,\del>0$ and find a $\del$-cover $\{E_n\}$ of $X$ such that $\sum_nh(dE_n)<\hm^h(X)+\eps$.
For each $n$ let $\mc B_n$ be an $dE_n$-cover of $Y$ such that $\abs{\mc B_n}=N_Y(dE_n)$.
If $\del$ is small enough, we thus have $\abs{\mc B_n}g(dE_n)<\ubox_0^g(Y)+\eps$.
Consider the family $\mc A=\{E_n\times B:n\in\Nset,B\in\mc B_n\}$. It is clearly a $\del$-cover
of $X\times Y$ and routine calculation shows that
$\sum_{A\in\mc A}h(dA)g(dA)\leq(\ubox_0^g(Y)+\eps)(\hm^h(X)+\eps)$.
Thus $\hm_\del^{hg}(X\times Y)\leq(\ubox_0^g(Y)+\eps)(\hm^h(X)+\eps)$. Let $\del\to0$
and $\eps\to0$ to get (ii).

(iii) comes from~\cite{MR1362951}.
(iv) is easily derived from the following
generalization of (iii) that can be found in~\cite{howroydPhD}, see also~\cite{MR1362951,Kelly}:
If $E\subs X\times Y$ and $g,h$ are of finite order, then
$\int_X\hm^g(E\cap\{x\}\times Y)\mathrm{d}\hm^h(x)\leq\hm^{gh}(E)$.

(v) and (vi): Let $C_X(\del)$ be the maximal number of points in $X$ that are pairwise more than
$\del$ apart. This is a variation of the covering number function $N_X(\del)$
and it is obvious that $C_X(\del)\leq N_X(\del)\leq C_X(\del/2)$.
So if $\tau^g$ and $\underrightarrow{\tau^g}$ are the set functions
that obtain the same way as $\ubox^g$ and $\dbox^g$, respectively,
from $C_X$ in place of $N_X$, it is clear that
$\tau^g\leq\ubox^g\leq\tau^{g_2}$, where $g_2(r)=g(2r)$, and likewise
$\underrightarrow{\tau^g}\leq\dbox^g\leq\underrightarrow{\tau}^{g_2}$.
Thus if $g,h$ are of finite order, there is a constant $q$ such that
$\tau^g\leq\ubox^g\leq q\tau^g$ and
$\underrightarrow{\tau^g}\leq\dbox^g\leq q\underrightarrow{\tau^g}$.
As proved in~\cite{ZinPack},
$\tau^h(X)\,\underrightarrow{\tau^g}(Y)\leq\tau^{hg}(X\times Y)$ and
$\underrightarrow{\tau^h}(X)\,\underrightarrow{\tau^g}(Y)
\leq\underrightarrow\tau^{hg}(X\times Y)$.
Hence (v) and (vi) hold with $c=1/q^2$.
\end{proof}
\subsection*{Uniformly continuous and Lipschitz images}
The following lemma on Lipschitz images and its counterpart for uniformly continuous mappings
are well-known for Hausdorff measures, see e.g.~\cite[Theorem 29]{MR0281862}.
\begin{lem}\label{lipschitz}
Let $f:X\to Y$ be a Lipschitz mapping with Lipschitz constant $L$.
Then $\hm^s(f(X))\leq L^s\hm^s(X)$ for any $s>0$.
Analogous statements hold also for $\uhm^s,\ubox^s$ and $\dbox^s$.
\end{lem}
\begin{lem}\label{lipschitz2}
Let $f:(X,d_X)\to (Y,d_Y)$ be a uniformly continuous mapping. Suppose $g\in\HH$
and that $f$ satisfies the condition
\begin{equation}\label{lip2}
  d_Y(f(x),f(y))\leq g(d_X(x,y)),\quad x,y\in X.
\end{equation}
Then $\hm^h(f(X))\leq \hm^{h{\circ}g}(X)$ for any $h\in\HH$.
Analogous statements hold also for
$\uhm^h$, $\ubox^h$ and $\dbox^h$.
\end{lem}

\subsection*{Dimensions}
Recall that the \emph{Hausdorff dimension} of $X$ is defined by
$$
  \hdim X=\sup\{s>0:\hm^s(X)=\infty\}=\inf\{s>0:\hm^s(X)=0\}.
$$

The \emph{upper Hausdorff dimension} arising from the upper Hausdorff measure
is defined by the same pattern:
$$
  \uhdim X=\sup\{s>0:\uhm^s(X)=\infty\}=\inf\{s>0:\uhm^s(X)=0\}.
$$
It is clear that $\hdim X\leq\uhdim X$.
It follows from Lemma~\ref{lem1}(v) that if $X$ is a complete metric space and $E\subs X$,
then $\uhdim E=\inf\{\hdim K:K\supseteq E\text{ is \si compact}\}$.
In particular, if $X$ is \si compact, then $\hdim X=\uhdim X$.

The \emph{packing dimension} obtains by the same pattern from $\ubox^s$:
$$
  \pdim E=\inf\{s>0:\ubox^s(E)=0\}=\sup\{s>0:\ubox^s(E)=\infty\}.
$$

The \emph{lower directed packing dimension} related to $\dbox^s$ is defined by
$$
  \dpdim X=\sup\{s>0:\dbox^s(X)=\infty\}=\inf\{s>0:\dbox^s(X)=0\}.
$$
The chain of inequalities~\eqref{basicIneq} yields~\eqref{basicinequality}.
\subsection*{The measures on the Cantor cube}
For $p\in2^{<\Nset}$ we denote $[p]=\{x\in\Cset:p\subs x\}$.
Metrize $\Cset$ as follows: Given $x\neq y\in\Cset$,
set $n(x,y)=\min\{i\in\Nset:x(i)\neq y(i)\}$ and define
$d(x,y)=2^{-n(x,y)}$.
This is a variant of the usual least difference metric on $\Cset$. In particular, the topology
induced by $d$ coincides with that of $\Cset$.

Routine proofs show that in this metric, $\hm^1$ coincides on Borel sets with the usual product
measure, i.e.~the Haar measure of the compact group $\Cset$, and that
$$
  \hm^1(\Cset)=\uhm^1(\Cset)=\ubox^1(\Cset)=\dbox^1(\Cset)=1.
$$

Besides the \si ideal $\EE$ generated by closed null sets we also introduce
the following family of highly regular compact subsets of $\Cset$.
For each $I\in[\Nset]^\Nset$ put
$C_I=\{x\in\Cset:x\rest I\equiv0\}$
and define $\CC=\{C_I:I\in[\Nset]^\Nset\}$.
\begin{lem}\label{EC}
\begin{enum}
\item $E\in\EE$ if and only if there is $h\prec 1$ such that $E\in\NNs(\ubox^h_0)$,
\item $\CC\subs\EE$,
\item for each $h\prec 1$ there is $C\in\CC$ such that $\hm^h(C)>0$.
\end{enum}
\end{lem}
\begin{proof}
(i) According to Lemma~\ref{lemP}(v) it is enough to prove that if $E\subs\Cset$
is a closed set and $\hm^1(E)=0$, then $\ubox^1_0(E)=0$.
Let $\eps>0$ be arbitrary. As $E$ is compact and $\hm^1(E)=0$, there
is a finite cover $\mc A$ of $E$ such that $\sum_{A\in\mc A}dA<\eps$.
Since any subset of $\Cset$ is a subset of a cylinder with the same diameter,
we may assume all $A\in\mc A$ are cylinders, so let $\mc A=\{[p_1],[p_2],\dots,[p_k]\}$.
Let $n\in\Nset$ be arbitrary subject to $n\geq\max\{\abs{p_1},\abs{p_2},\dots,\abs{p_k}\}$.
Let $\mc B=\{p\in2^n:\exists i\leq k(p_i\subs p)\}$.
It is clear that $\mc B$ is a $2^{-n}$-cover of $E$. Therefore $N_E(2^{-n})\leq\abs{\mc B}$.
It is also clear that
$\sum_{A\in\mc A}dA=\sum_{B\in\mc B}dB=2^{-n}\cdot\abs{\mc B}$.
Consequently $2^{-n}\cdot N_E(2^{-n})<\eps$.
Since this is true for all $n\in\Nset$ large enough, we get $\ubox_0^1(E)\leq\eps$.
Letting $\eps\to 0$ yields $\ubox_0^1(E)=0$.

(ii) Let $I\in[\Nset]^\Nset$. For each $n\in\Nset$,
the family $\{[p]:p\in C_I\rest n\}$ is obviously a $2^{-n}$-cover
of $C_I$ of cardinality $2^{\abs{n\setminus I}}$. Therefore
$\hm^1_{2^{-n}}(C_I)\leq2^{\abs{n\setminus I}}2^{-n} =2^{-\abs{n\cap I}}$.
Hence $\hm^1(C_I)\leq\lim_{n\to\infty}2^{-\abs{n\cap I}}=0$.

(iii)
Using Lemma~\ref{lemHaus}(i) it is enough to find $C_I$ such that $\hm^h(C_I)\geq 1$.
$h\prec 1$ yields  $\frac{h(2^{-n})}{2^{-n}}\to\infty$. Therefore there is
$I\in[\Nset]^\Nset$ sparse enough to satisfy
$2^{\abs{n\cap I}}\leq\frac{h(2^{-n})}{2^{-n}}$,
i.e.~$2^{-\abs{n\setminus I}}\leq h(2^{-n})$ for all $n\in\Nset$.
Consider the product measure on $C_I$ given as follows: If $p\in2^n$ and
$[p]\cap C_I\neq\emptyset$, put
$\lambda([p]\cap C_I)=2^{-\abs{n\setminus I}}$.
Straightforward calculation shows that $h(dE)\geq\lambda(E)$ for each $E\subs C_I$.
Hence $\sum_n h(dE_n)\geq\sum_n\lambda(E_n)\geq\lambda(C_I)=1$
for each cover $\{E_n\}$ of $C_I$ and $\hm^h(C_I)\geq 1$ follows.
\end{proof}

\section{\hnull{} spaces}
\label{sec:hnull}
We first establish a couple of characterizations of \hnull{} spaces in terms of
Hausdorff measures and dimensions.
\begin{thm}\label{basicHnull}
Let $X$ be a metric space. The following are equivalent.
\begin{enum}
\item $X$ is \hnull,
\item $\hdim f(X,\rho)=0$ for each uniformly equivalent metric on $X$,
\item $\hdim f(X)<\infty$ for each uniformly continuous $f:X\to Y$ into another
metric space,
\item $\hm^h(X)=0$ for each $h\in\HH$,
\item $\hm^h(X)<\infty$ for each $h\in\HH$.
\end{enum}
\end{thm}
\begin{proof}
Denote by $d$ the metric of $X$.
(i)\Implies(ii) and (i)\Implies(iii) are trivial.
We first prove simultaneously (ii)\Implies(iv) and (iii)\Implies(iv).
Let $h\in\HH$. Choose a convex Hausdorff function $g$ such that $g\prec h^{1/n}$ for each
positive $n\in\Nset$.
The properties of $g$ ensure that
$\rho(x,y)=g(d(x,y))$
is a uniformly equivalent metric on $X$. The identity
map $\id_X:(X,d)\to(X,\rho)$ is of course uniformly continuous.
Thus if either (ii) or (iii) holds, then $\hdim(X,\rho)<\infty$,
hence there is $n$ such that $\hm^n(X,\rho)=0$.
The choice of $g$ ensures $h{\circ}g^{-1}\succ n$.
Thus $\hm^{h{\circ}g^{-1}}(X,\rho)=0$ by Lemma~\ref{lemHaus}(i).
Also $d(x,y)\leq g^{-1}(\rho(x,y))$. Hence Lemma \ref{lipschitz2} yields
$\hm^h(X,d)\leq\hm^{h{\circ}g^{-1}}(X,\rho)=0$.

(iv)\Implies(v) is trivial.

(v)\Implies(i):
Let $f:X\to (Y,\rho)$ be uniformly continuous.
There is a function $g\in\HH$ such that
$\rho(fx,fy)\leq g(d(x,y))$ for all $x,y\in X$.
Fix $s>0$ and put $h(r)=r^s$. By assumption $\hm^{h{\circ}g}(X)<\infty$.
Apply Lemma~\ref{lipschitz2} to conclude that $\hm^s(fX)=\hm^h(fX)\leq\hm^{h{\circ}g}(X)<\infty$.
In particular $\hdim f(X)\leq s$.
As $s>0$ was arbitrary, it follows that $\hdim f(X)=0$.
\end{proof}
Let $X$ be a space and let $\seq{U_n}$ be a sequence of
subsets of $X$. Given a sequence $\Seqeps$ of positive real numbers,
$\seq{U_n}$ is termed $\seqeps$-fine if $dU_n\leq\eps_n$
holds for all $n$.
Recall once again that $X$ has strong measure zero (\smz{})
if for any $\Seqeps$ there is an $\seqeps$-fine cover of $X$.
\begin{lem}\label{seqep2}
For each $\Seqeps$ there exists $g\in\HH$ such that: If $\hm^g(X)=0$,
then $X$ admits an $\seqeps$-fine $\lambda$-cover.
\end{lem}
\begin{proof}
Assume without loss of generality that $X$ has no isolated points.
Let$\seq{\eps_n}\in(0,\infty)^\Nset$.
Choose $h\in\HH$ such that $h(\eps_n)\geq\frac1n$ for all $n\in\Nset$.
Suppose $\hm^h(X)=0$. Lemma~\ref{lambda} yields a $\lambda$-cover $\{G_n\}$
such that $\sum_nh(dG_n)<\infty$.
Reordering we may assume $\sum_nh(dG_n)<1$ and $dG_0\geq dG_1\geq dG_2\geq\dots$.
Therefore
$$
 h(dG_n)\leq\frac1n\, nh(dG_n)\leq\frac1n\sum_nh(dG_n)<\frac1n\leq h(\eps_n)
$$
and since $h$ is nondecreasing, we get $dG_n\leq\eps_n$.
\end{proof}
\begin{thm}\label{hnullsmz}
A metric space $X$ is \hnull{} if and only if it is \smz{}.
\end{thm}
\begin{proof}
The forward implication follows at once from the above lemma.
To prove the reverse one, let $h\in\HH$, fix $\del>0$ and choose $\eps_n<\del$ to satisfy
$h(\eps_n)\leq 2^{-n}$. The $\seq{\eps_n}$-fine cover of $X$ witnesses $\hm_\del^h(X)\leq 1$.
Therefore $\hm^h(X)\leq1$, which is by Theorem~\ref{basicHnull}(v) enough.
\end{proof}
\hnull{}$\,=\,$\smz{} sets are characterized by behavior of cartesian products.
\begin{thm}\label{prodHnull}
The following are equivalent.
\begin{enum}
\item $X$ is \hnull{},
\item $\hm^h(X\times Y)=0$ whenever $h\in\HH$ and $\uhm_0^h(Y)=0$,
\item $\hm^h(X\times Y)=0$ whenever $h\in\HH$, $Y$
is \si compact and $\hm^h(Y)=0$,
\item $\hm^1(X\times E)=0$ whenever $E\in\EE$,
\item $\hm^1(X\times C)=0$ whenever $C\in\CC$.
\end{enum}
\end{thm}
\begin{proof}
(i)\Implies(ii):
Suppose $X$ is \hnull{}. Fix $\eta>0$.
Since $\uhm_0^h(Y)=0$, for each $j\in\Nset$ there is a finite
family $\mc U_j$ of sets such that $\sum_{U\in\mc U_j}h(dU)<2^{-j}\eta$.
We may also assume that $dU<\eta$ for all $U$.

Let $\eps_j=\min\{dU:U\in\mc U_j\}$.
Choose a cover $\{V_j\}$ of $X$ such that $dV_j\leq\eps_j$ and
define
$$
  \mc W=\{V_j\times U:j\in\Nset,\,U\in\mc U_j\}.
$$
It is obvious that $\mc W$ is a cover of $X\times Y$. Since
$d(V_j\times U)=d(U)$ for all $j$ and $U\in\mc U_j$ by the choice
of $\eps_j$, we have
$$
  \sum_{W\in\mc W}h(dW)=
  \sum_{j\in\Nset}\sum_{U\in\mc U_j}h(dU)<
  \sum_{j\in\Nset}2^{-j}\eta=2\eta.
$$
Therefore $\hm_\eta^h(X\times Y)<2\eta$, which is enough for
$\hm^h(X\times Y)=0$, as $\eta$ was arbitrary.

(ii)\Implies(iii)\Implies(iv)\Implies(v) is trivial.

(v)\Implies(i):
Suppose $X$ is not \hnull. We will show that $\hm^1(X\times C)>0$
for some $C\in\CC$.
By assumption there is $h\in\HH$ such that $\hm^h(X)>0$.
\emph{Mutatis mutandis} we may assume $h$ be concave and $h(r)\geq\sqrt r$.
In particular $g(r)=r/h(r)$ is an increasing function and
$\lim_{r\to0}g(r)=0$, i.e.~$g$ is Hausdorff function, and $g\prec1$.
Also both $h$ and $g$ are of finite order.
Use Lemma~\ref{EC}(iii) to find $C\in\CC$ such that $\hm^g(C)>0$.
Now apply Lemma~\ref{howroyd}(iii):
$$
  \hm^1(X\times C)=\hm^{h\cdot g}(X\times C)\geq\hm^h(X)\cdot\hm^g(C)>0.
  \qedhere
$$
\end{proof}
\begin{coro}
If $X$ is \hnull{} then $\hdim X\times Y=\hdim Y$ for every \si compact metric space $Y$.
\end{coro}

\section{\uhnull{} spaces}
\label{sec:uhnull}
\begin{thm}\label{basicUhnull}
The following are equivalent.
\begin{enum}
\item $X$ is \uhnull,
\item $\uhdim f(X,\rho)=0$ for each uniformly equivalent metric on $X$,
\item $\uhdim f(X)<\infty$ for each uniformly continuous $f:X\to Y$ into another
metric space,
\item $X\in\NNs(\uhm_0^h)$ for each $h\in\HH$,
\item $\uhm^h(X)<\infty$ for each $h\in\HH$.
\end{enum}
\end{thm}
\begin{proof}
One has to employ Lemma~\ref{lem1} instead of Lemma~\ref{lemHaus},
otherwise the proof goes the same way as that of Theorem~\ref{basicHnull}.
\end{proof}
Next we provide a combinatorial characterization of \uhnull{} sets
that parallels Theorem~\ref{hnullsmz}.
\begin{lem}\label{seqep}
For each $\Seqeps$ there exists $g\in\HH$ such that: If $\uhm^g(X)=0$,
then $X$ admits an $\seqeps$-fine $\gamma$-groupable cover.
\end{lem}
\begin{proof}
Assume without loss of generality that $X$ has no isolated points.
Choose Hausdorff functions $g,h$ such that $h(\eps_n)\geq\frac1n$ for all $n\in\Nset$
and $g\prec h$.
Suppose $\uhm^g(X)=0$. Then $X\in\NNs(\uhm_0^h)$
by Lemma~\ref{lem1}(vii). By Lemma~\ref{gammagr} there is a
$\gamma$-groupable cover $\{G_n\}$ such that $\sum_nh(dG_n)<\infty$.
Proceed as in the proof of Lemma~\ref{seqep2}.
\end{proof}

\begin{thm}\label{combUhnull}
Let $X$ be a separable metric space.
$X$ is \uhnull{} if and only if for each $\Seqeps$, $X$ has
an $\seqeps$-fine $\gamma$-groupable cover.
\end{thm}
\begin{proof}
The forward implication follows at once from the above lemma.
The reverse implication is proved as the one of Theorem~\ref{hnullsmz}.
\end{proof}

\begin{thm}\label{prodUhnull}
The following are equivalent.
\begin{enum}
\item $X$ is \uhnull,
\item for each $h\in\HH$, $Y\in\NNs(\uhm_0^h)$ and each complete $Z\sups X$
there is $F$ \si compact, $X\subs F\subs Z$,
such that $\uhm^h(F\times Y)=0$,
\item $\uhm^h(X\times Y)=0$ whenever $h\in\HH$ and $\uhm_0^h(Y)=0$,
\item $\uhm^1(X\times E)=0$ for each $E\in\EE$,
\item $\uhm^1(X\times C)=0$ for each $C\in\CC$.
\end{enum}
\end{thm}
\begin{proof}
The proof is similar to that of Theorem~\ref{prodHnull}. The only nontrivial implications
are
(i)\Implies(ii) and (v)\Implies(i).

(i)\Implies(ii):
Let $Z\sups X$ be a complete metric space.
Suppose $X$ is \uhnull{}. In particular, by Lemma~\ref{lem1} $X$ is contained in
a \si compact set $K\subs Z$. Let $h\in\HH$ and $Y\in\NNs(\uhm_0^h)$.
Lemma~\ref{gammagr} yields a $\gamma$-groupable cover $\mc U$ of $Y$ such that
$\sum_{U\in\mc U}h(dU)<\infty$.
Denote by $\mc U_j$ the witnessing groups.
Let $\eps_j=\min\{dU:U\in\mc U_j\}$.
Using Theorem~\ref{combUhnull} choose a $\gamma$-groupable cover $\{V_j\}$ of $X$
such that $dV_j\leq\eps_j$.  We may and shall assume that each $V_j$ is a closed subset of $Z$.
Denote by $\mc V_k$ the witnessing groups.
Define
\begin{align*}
  \mc W&=\{V_j\times U:j\in\Nset,\,U\in\mc U_j\},\\
  F&=K\cap\bigcup_{i\in\Nset}\bigcap_{k\geq i}\bigcup\mc V_k.
\end{align*}
The set $F\subs Z$ is clearly an $F_\sigma$ subset of $K$ and is thus \si compact.
It is easy to check that $\mc W$ is a $\gamma$-groupable cover of $F\times Y$.
Since $d(V_j\times U)=d(U)$ for all $j$ and $U\in\mc U_j$ by the choice
of $\eps_j$, we have
$$
  \sum_{W\in\mc W}h(dW)=
  \sum_{U\in\mc U}h(dU)<\infty.
$$
Using Lemma~\ref{gammagr} it follows that $F\times Y\in\NNs(\uhm_0^h)$ and
in particular $\uhm^h(X\times Y)=0$.

(v)\Implies(i):
Suppose $X$ is not \uhnull. We will show that $\hm^1(X\times C)>0$
for some $C\in\CC$.
By assumption there is $h\in\HH$ such that $\uhm^h(X)>0$. As well as in the proof of
Theorem~\ref{prodHnull} suppose $h$ is concave, hence of finite order, and
find a Hausdorff function of finite order $g\prec1$ such
that $g(r)h(r)=r$ and $C\in\CC$ such that $\hm^g(C)>0$. This time apply
Lemma~\ref{howroyd}(iv):
$$
  \uhm^1(X\times C)=\uhm^{h\cdot g}(X\times C)\geq\uhm^h(X)\cdot\uhm^g(C)>0.
  \qedhere
$$
\end{proof}
\begin{coro}
If $X$ is \uhnull{} then $\uhdim X\times Y=\uhdim Y$ for every metric space $Y$.
In particular, $\uhdim X\times Y=\hdim Y$ if $Y$ is \si compact.
\end{coro}

\subsection*{Products of \hnull{} and \uhnull{} sets}
It is well known that a product of two \smz{} sets need not be \smz{}.
Thus the product of two \hnull{} sets need not be \hnull{}.
But if one of the factors is \uhnull{}, the product is \hnull{}:
\begin{thm}\label{productHUH}
\begin{enum}
\item If $X$ and $Y$ are \uhnull{}, then $X\times Y$ is \uhnull{}.
\item If $X$ is \hnull{} and $Y$ is \uhnull{}, then $X\times Y$ is \hnull{}.
\end{enum}
\end{thm}
\begin{proof}
(i) Suppose $X,Y$ are \uhnull{}. By Theorem~\ref{basicUhnull}(iv)
$Y\in\NNs(\uhm_0^h)$ for all $h\in\HH$. Hence Theorem~\ref{prodUhnull}(ii) yields
$\uhm^h(X\times Y)=0$ for all $h\in\HH$, which is by Theorem~\ref{basicUhnull}(v) enough.

(ii) is derived in the same manner from Theorems~\ref{prodHnull} and~\ref{basicHnull}.
\end{proof}

M.~Scheepers~\cite[Theorem 1, Lemma 3]{MR1779763} proved that a product of two \smz{} sets
is \smz{} as long as one of the sets satisfies the Hurewicz property.
Theorem~\ref{productHUH}(ii) improves his result (recall that \smz{}$=$\hnull{}).
Indeed, it is easy to show that a \smz{} space satisfying the Hurewicz property
is \uhnull{}. But since \uhnull{} is a uniform property and
Hurewicz property is topological, one cannot expect \emph{a priori} that
every \uhnull{} set has the Hurewicz property. (It is of course so if Borel Conjecture holds.)
\begin{prop}
Assuming the Continuum Hypothesis, there is an \uhnull{} set that does not have the Hurewicz property.
\end{prop}
\begin{proof}
It follows from~\cite[Theorem 1]{MR738943} and its proof that under
the Continuum Hypothesis there is a $\gamma$-set
$X\subs\Cset$ that is concentrated on a countable set $D$.
By~\cite[Theorem 6]{MR738943}, every
$\gamma$-set is $\MM$-additive. By Theorem~\ref{mainME} \emph{infra},
every $\MM$-additive set is \uhnull{}. Hence $X$ is \uhnull{}. On the other hand,
as proved in~\cite[Theorem 20]{MR1610427}, the set $X\setminus D$ does not have the Hurewicz property
and since it is a subset of $X$, it is clearly \uhnull{}.
\end{proof}
%
%
%
%

\subsection*{Universally meager sets}
Recall that a separable metric space $E$ is termed \emph{universally meager}
\cite{MR1814112,MR2427418}
if for any perfect Polish spaces $Y,X$ such that $E\subs X$ and every
continuous one--to--one mapping $f:Y\to X$ the set $f^{-1}(E)$ is meager
in $Y$. We show that \uhnull{} sets are universally meager.
\begin{lem}\label{lemUM}
Let $X,Y,Z$ be perfect Polish spaces and $\phi:Y\to X$ a continuous one--to--one mapping.
Let $\mc F$ be an equicontinuous family of uniformly continuous mappings of $Z$ into $X$.
If $E\subs Z$ is \uhnull{}, then there is a \si compact set
$F\sups E$ such that $\phi^{-1}f(F)$ is meager in $Y$ for all $f\in\mc F$.
\end{lem}
\begin{proof}
Let $\{U_n\}$ be a countable base for $Y$.
As $\phi$ is one--to--one the set $\phi(U_n)$ is analytic and
uncountable for each $n$. Therefore it contains a perfect set and thus is not \uhnull{},
i.e.~there is $h_n\in\HH$ such that $\uhm^{h_n}(\phi U_n)>0$.
Choose $h\in\HH$ such that $h\prec h_n$ for all $n$, so that
$\uhm^h(\phi U_n)>0$ for all $n$.
Therefore $\uhm^h(\phi U)>0$ for each nonempty set $U$ open in $Y$.

Since $\mc F$ is equicontinuous, there is a function $g\in\HH$
such that~\eqref{lip2} is satisfied by each $f\in\mc F$.
By Theorem~\ref{basicUhnull} $E\in\NNs(\uhm_0^{h{\circ}g})$. Therefore there is a \si compact set
$F\sups E$ such that $\uhm^{h{\circ}g}(F)=0$. Hence Lemma~\ref{lipschitz2} guarantees that
$\uhm^h(fF)=0$ for all $f\in\mc F$. Therefore the $F_\sigma$-set $\phi^{-1}f(F)$
is meager in $Y$: for otherwise
it would contain an open set witnessing $\uhm^h(fF)>0$.
\end{proof}
\begin{thm}\label{umg}
Every \uhnull{} set is universally meager.
\end{thm}
\begin{proof}
Apply Lemma~\ref{lemUM} with $Z=X$ and $\mc F=\{\id_X\}$.
\end{proof}
\section{\upnull{} spaces}
\label{sec:pnull}

\begin{thm}\label{basicPnull}
The following are equivalent.
\begin{enum}
\item $X$ is \upnull,
\item $\updim f(X,\rho)=0$ for each uniformly equivalent metric on $X$,
\item $\updim f(X)<\infty$ for each uniformly continuous $f:X\to Y$ into another
metric space,
\item $X\in\NNs(\ubox_0^h)$ for each $h\in\HH$,
\item $\ubox^h(X)<\infty$ for each $h\in\HH$.
\end{enum}
\end{thm}
\begin{proof}
This is proved in the same manner as Theorems~\ref{basicHnull} and~\ref{basicUhnull}, with the aid
of Lemma~\ref{lemP}.
\end{proof}
Next we provide a combinatorial characterization of \upnull{} sets.
Note the similarity of \ref{combPnull}(ii)
with Theorem~\ref{combUhnull}.
\begin{lem}\label{seqep3}
For each $\Seqeps$ there exists $g\in\HH$ such that: If $\ubox^g(X)=0$,
then $X$ admits an $\seqeps$-fine $\gamma$-groupable cover such that
the witnessing groups $\mc G_j$ satisfy $\abs{\mc G_j}\leq j$ for each $j$.
\end{lem}
\begin{proof}
Assume without loss of generality that $\seqeps$ is decreasing.
Set $\del_n=\eps_{0+1+\dots+n}$.
Choose a Hausdorff function $g$ such that $g(\del_n)>\frac1n$ for all $n\in\Nset$.
Suppose $\ubox^g(X)=0$. Use Lemma~\ref{IncreasingSetsLemma} to find $X_k\upto X$
such that $\ubox_0(X_k)<1$.
Thus for each $k$ there is $n_k$ such that
$N_{X_k}(\del_n)g(\del_n)<1$ for each $n\geq n_k$.
Define the witnessing groups as follows:
If $n_k\leq j<n_{k+1}$, let $\mc G_j$ be a cover of $X_k$ witnessing $N_{X_k}(\del_j)g(\del_j)<1$.
Clearly $\abs{\mc G_j}<\frac{1}{g(\del_j)}<j$.
The cover we are looking for is of course $\bigcup_{j\in\Nset}\mc G_j=\{U_n:n\in\Nset\}$
ordered in such a way that if $i<j$ and $U_n\in\mc G_i$ and $U_m\in\mc G_j$, then $n<m$.
It is clear that $\{U_n\}$ is a $\gamma$-groupable cover.
If $U_n\in\mc G_j$, then $n\leq\sum_{i\leq j}\abs{\mc G_i}<0+1+\dots+j$. Hence
$dU_n<\del_j=\eps_{0+1+\dots+j}<\eps_n$. Therefore $\{U_n\}$ is $\seqeps$-fine.
\end{proof}
\begin{thm}\label{combPnull}
Let $X$ be a separable metric space. The following are equivalent.
\begin{enum}
\item $X$ is \upnull{},
\item for each $\Seqeps$, $X$ has
an $\seqeps$-fine $\gamma$-groupable cover such that the witnessing groups $\mc G_j$
satisfy $\abs{\mc G_j}\leq j$ for each $j$,
\item for each $\Seqeps$, there is a sequence $\{\mc F_n\}$ of families
of sets such that $d\mc F_n\leq\eps_n$ and $\abs{\mc F_n}\leq n$ for all $n\in\Nset$
and the sequence $\{\bigcup\mc F_n\}$ is a $\gamma$-cover of $X$,
\item there is $f\in\Pset$ such that for each $\Seqeps$, there is a sequence $\{\mc F_n\}$
of families
of sets such that $d\mc F_n\leq\eps_n$ and $\abs{\mc F_n}\leq f(n)$ for all $n\in\Nset$
and the sequence $\{\bigcup\mc F_n\}$ is a $\gamma$-cover of $X$.
\end{enum}
\end{thm}
\begin{proof}
(i)\Implies(ii) follows from the above lemma. (ii)\Implies(iii)\Implies(iv) is trivial.
(iv)\Implies(i): Let $h$ be a Hausdorff function.
Choose $\seqeps$ so that $h(\eps_n)<\frac1{f(n+1)}$. Consider the families $\mc F_n$ given by (iii)
and for each $k$ set $X_k=\{x\in X:\forall n\geq k\ x\in\bigcup\mc F_n\}$.
Clearly $X_k\upto X$. It remains to show $\ubox_0^h(X_k)\leq 1$.
Let $\eps<\eps_k$. There is $n>k$ such that $\eps_n\leq\eps<\eps_{n-1}$.
Obviously $N_{X_k}(\eps)\leq N_{X_k}(\eps_n)\leq\abs{\mc F_n}\leq f(n)$. Hence
$N_{X_k}(\eps)h(\eps)\leq f(n)h(\eps_{n-1})\leq\frac{f(n)}{f(n-1+1)}=1$.
Thus $\ubox_0^h(X_k)\leq1$.
\end{proof}
\begin{thm}\label{prodPnull}
The following are equivalent.
\begin{enum}
\item $X$ is \upnull,
\item for each $h\in\HH$, $Y\in\NNs(\ubox_0^h)$ and each complete $Z\sups X$
there is $F$ \si compact, $X\subs F\subs Z$,
such that $\ubox^h(F\times Y)=0$,
\item $\ubox^h(X\times Y)=0$ whenever $h\in\HH$ and $\ubox_0^h(Y)=0$,
\item $\ubox^1(X\times E)=0$ for each $E\in\EE$,
\item $\ubox^1(X\times C)=0$ for each $C\in\CC$.
\end{enum}
\end{thm}
\begin{proof}
(i)\Implies(ii):
Suppose $Y\in\NNs(\ubox_0^h)$. By Lemma~\ref{lemP}(v) there is $g\prec h$ such that
$Y\in\NNs(\ubox_0^h)$. Let $f\in\HH$ be such that $fg\geq h$. Since $X$ is \upnull{},
Theorem~\ref{basicPnull}(iv) yields $X\in\NNs(\ubox_0^f)$. Thus there is, by Lemma~\ref{lemP}(iii),
a \si compact set $F\sups X$ such that $\ubox^f(F)=0$.
Apply Lemma~\ref{howroyd}(i):
$$
  \ubox^h(F\times Y)\leq\ubox^{fg}(F\times Y)\leq\ubox^f(F)\ubox^g(Y)=0.
$$

(ii)\Implies(iii)\Implies(iv)\Implies(v) is trivial.

(v)\Implies(i):
Suppose $X$ is not \upnull{}. We will show that $\ubox^1(X\times C)>0$
for some $C\in\CC$.
By assumption there is $h\in\HH$ such that $\ubox^h(X)>0$.
Proceed as in the proof of
Theorem~\ref{prodHnull} this time applying Lemma~\ref{howroyd}(v).
\end{proof}
\begin{thm}
If $X$ and $Y$ are \upnull{}, then $X\times Y$ is \upnull{}.
\end{thm}
\begin{proof}
Apply Theorem~\ref{basicPnull}(iv) and (v) and Theorem~\ref{prodPnull}(iii).
\end{proof}
\begin{prop}\label{XN}
If $X$ is \upnull{}, then
\begin{enum}
\item $\hm^h(X\times Y)=0$ whenever $h\in\HH$ and $\hm^h(Y)=0$,
\item in particular $\hm^1(X\times N)=0$ for each $N\in\NN$.
\end{enum}
\begin{proof}
(i) follows from Theorem~\ref{basicPnull}(v) and Lemma~\ref{howroyd}(ii).
(ii) follows from (i).
\end{proof}
\end{prop}
\begin{coro}
If $X$ is \upnull{} then for every metric space $Y$
\begin{enum}
\item $\pdim X\times Y=\pdim Y$,
\item $\hdim X\times Y=\hdim Y$.
\end{enum}
\end{coro}

\section{\Dpnull{} spaces}
\label{sec:dpnull}

\begin{thm}\label{basicDpnull}
The following are equivalent.
\begin{enum}
\item $X$ is \dpnull,
\item $\dpdim f(X,\rho)=0$ for each uniformly equivalent metric on $X$,
\item $\dpdim f(X)<\infty$ for each uniformly continuous $f:X\to Y$ into another
metric space,
\item $\dbox^h(X)=0$ for each $h\in\HH$,
\item $\dbox^h(X)<\infty$ for each $h\in\HH$.
\end{enum}
\end{thm}
\begin{proof}
This is proved in the same manner as Theorems~\ref{basicHnull} and~\ref{basicUhnull},
with the aid of Lemma~\ref{lemDP}.
\end{proof}
Note the similarity of the following characterization of \dpnull{}-sets
with Theorem~\ref{combPnull}.
\begin{thm}
Let $X$ be a separable metric space. The following are equivalent. \label{combDnull}
\begin{enum}
\item $X$ is \dpnull{},
\item for each $\Seqeps$, there is $I\in[\Nset]^\Nset$ and a sequence
$\{\mc F_n:n\in\Nset\}$ of families of sets such that $d\mc F_n\leq\eps_n$ and
$\abs{\mc F_n}\leq n$ for all $n\in I$
and the sequence $\{\bigcup\mc F_n:n\in I\}$ is a $\gamma$-cover of $X$,
\item there is $f\in\Pset$ such that for each $\Seqeps$,
there is $I\in[\Nset]^\Nset$ and a sequence
$\{\mc F_n:n\in\Nset\}$ of families of sets such that $d\mc F_n\leq\eps_n$ and
$\abs{\mc F_n}\leq f(n)$ for all $n\in I$
and the sequence $\{\bigcup\mc F_n:n\in I\}$ is a $\gamma$-cover of $X$.
\end{enum}
\end{thm}
\begin{proof}
(i)\Implies(ii):
Assume without loss of generality that $\seqeps$ is decreasing.
Choose a Hausdorff function $g$ such that $g(\eps_n)=\frac1n$ for all $n>0$.
Suppose $\dbox^g(X)=0$. Thus there is a sequence $X_k\upto X$
such that $\lbox_0(X_k)<\frac12$.
Therefore it is possible to choose a decreasing sequence $\del_k>0$ such that
$N_{X_k}(\del_k)g(\del_k)<\frac12$.
Fix $k$ for the moment. Let $m(k)$ be the (unique) integer satisfying
$\eps_{m(k)+1}\leq\del_k<\eps_{m(k)}$.
Then
$$
  N_{X_k}(\eps_{m(k)})\leq
  N_{X_k}(\del_k)\frac{g(\del_k)}{g(\eps_{m(k)})}\leq
  \frac12(m(k)+1)\leq m(k).
$$
Choose a cover $\mc F_{m(k)}$ of $X_k$ witnessing $N_{X_k}(\del_k)g(\del_k)<\frac12$.
Let $I=\{m(k):k\in\Nset\}$. Verification of the required properties of
$\{\bigcup\mc F_n:n\in I\}$ is straightforward.

(ii)\Implies(iii) is trivial.
(iii)\Implies(i):
Let $g$ be a Hausdorff function.
Choose $\Seqeps$ decreasing subject to $g(\eps_n)\leq\frac{1}{f(n)}$.
Let $I$ and $\{\mc F_n:n\in\Nset\}$ be as in (iii). For $n\in I$ and $k\in\Nset$
set $F_n=\bigcup\mc F_n$ and $X_k=\bigcap_{k\leq n\in I}F_n$.
Obviously $X_k\upto X$ and
$$
  \forall n\geq k, n\in I\quad
  N_{X_k}(\eps_n)g(\eps_n)\leq
  \abs{\mc F_n}g(\eps_n)\leq f(n)\frac{1}{f(n)}=1,
$$
which yields $\lbox_0^g(X_k)\leq 1$. Consequently $\dbox^g(X)\leq1$. Thus
$X$ is \dpnull{} by Theorem~\ref{basicDpnull}(v).
\end{proof}
This combinatorial description of \dpnull{} sets yields a surprising consequence:
though $\dbox^g$ is not even finitely additive, \dpnull{} is a \si additive property.
\begin{coro}
For each metric space $X$, the family of all \dpnull{} subsets of $X$ forms a \si ideal.
\label{dpnullideal}
\end{coro}
\begin{proof}
This follows by a diagonal construction.
Let $\{X_k\}$ be a countable collection of \dpnull{} subsets of $X$ and $Y=\bigcup_nX_k$.
Let $\Seqeps$.
Apply repeatedly Theorem~\ref{combDnull}(ii) to find a diagonal sequence $\seq{n_i:i\in\Nset}$
and a triangular matrix $\{\mc F_{k,i}:k\leq i\in\Nset\}$ of collections of sets
with the following properties:
\begin{enumerate}[{\rm(a)}]
\item $\forall i\in\Nset\,\forall k\leq i\ d\mc F_{i,k}\leq\eps_{n_i}$,
\item $\forall i\in\Nset\,\forall k\leq i\ \abs{\mc F_{i,k}}\leq n_i$,
\item $\forall k\in\Nset\, \{\bigcup\mc F_{i,k}:i\geq k\}$ is a $\gamma$-cover of $X_k$.
\end{enumerate}
For each $i$ put $\mc G_i=\bigcup_{k\leq i}\mc F_{i,k}$.
Then (a) yields $d\mc G_i\leq\eps_{n_i}$ , (b) yields $\abs{\mc G_i}\leq n_i^2$, and (c) yields that
$\{\bigcup\mc G_i:i\in\Nset\}$ is a $\gamma$-cover of $Y$.
Apply Theorem~\ref{combDnull}(iii) with $f(n)=n^2$ to conclude that $Y$ is \dpnull{}.
\end{proof}
\begin{thm}
The following are equivalent.\label{prodDpnull}
\begin{enum}
\item $X$ is \dpnull,
\item for each $h\in\HH$, each $Y$ such that $\lbox_0^h(Y)=0$ and each complete $Z\sups X$
there is $F$ \si compact, $X\subs F\subs Z$,
such that $\dbox^h(F\times Y)=0$,
\item $\dbox^h(X\times Y)=0$ whenever $h\in\HH$ and $\lbox_0^h(Y)=0$,
\item $\dbox^1(X\times E)=0$ for each $E\in\EE$,
\item $\dbox^1(X\times C)=0$ for each $C\in\CC$.
\end{enum}
\end{thm}
\begin{proof}
(i)\Implies(ii):
Suppose $\lbox_0^h(Y)=0$. Choose $\seqeps$ such that $N_Y(\eps_n)h(\eps_n)<2^{-n}$.
Since $X$ is \dpnull{}, Theorem~\ref{combDnull}(ii) yields $I\in[\Nset]^\Nset$
and a sequence of sets $X_k\upto X$ such that $N_{X_k}(\eps_n)\leq n$ whenever $k\leq n\in I$.
Set $F=\bigcup_k \clos{X}_k$, the closures in $Z$. Since $X_k$'s are of finite box content,
they are by Lemma~\ref{lemP}(i) totally bounded. Since $Z$ is complete, their closures are compact.
Thus $F$ is \si compact.
Clearly $\clos{X}_k\times Y\upto F\times Y$.
Also $N_{\clos X_k\times Y}(\eps_n)\leq n2^{-n}$ whenever $k\leq n\in I$.
Hence $\lbox_0^h(\clos X_k\times Y)\leq\lim n2^{-n}=0$.

(ii)\Implies(iii)\Implies(iv)\Implies(v) is trivial and
(v)\Implies(i) is proved as in Theorem~\ref{prodPnull}, with the aid of
Lemma~\ref{howroyd}(vi).
\end{proof}

\begin{thm}
If $X$ and $Y$ are \dpnull{}, then $X\times Y$ is \dpnull{}.
\end{thm}
\begin{proof}
This follows from Theorems~\ref{basicDpnull} and~\ref{prodDpnull}(iii).
\end{proof}
\begin{coro}
If $X$ is \dpnull{} then for every metric space $Y$
$\dpdim X\times Y=\dpdim Y$.
\end{coro}
Recall that a separable metric space $X$ is a \emph{$\gamma$-set} if
each countable $\omega$-cover contains a subcover that is a $\gamma$-cover
(a cover $\mc U$ is an $\omega$-cover if for each finite set $F\subs X$ there
is a set $U\in\mc U$ such that $F\subs U$).
Nowik and Weiss~\cite{MR1905154} introduce a notion of a \Tprime-set
(cf.~Section~\ref{sec:TT}) and prove that every $\gamma$-set is \Tprime{}.
In view of Theorem~\ref{mainTT} \emph{infra},
the following generalizes their result.
\begin{prop}
Every $\gamma$-set is \dpnull{}.
\end{prop}
\begin{proof}
Let $X$ be a $\gamma$-set.
We verify condition (ii) of Theorem~\ref{combDnull}. Let $\Seqeps$.
Fix an infinite set $\{x_n:n\in\Nset\}\subs X$.
For $F\in[X]^{<\omega}$ put
$$
  U(F)=\bigcup\nolimits_{x\in F}B\bigl(x,\tfrac12\eps_{\abs F}\bigr)
  \setminus\{x_{\abs F}\}.
$$
The family $\{U(F):F\in[X]^{<\omega}\}$
is obviously an $\omega$-cover. Therefore there is a sequence $\{F_n\}$
of finite sets such that $\{U(F_n)\}$ is a $\gamma$-cover.
We may clearly assume that $\abs{F_0}\leq\abs{F_1}\leq\abs{F_2}\leq\dots$.
Since $U(F)$ misses $x_{\abs F}$, for each $k\in\Nset$ there are only
finitely many $n$'s such that $\abs{F_n}=k$. Passing to a subsequence
we may thus assume that $\abs{F_0}<\abs{F_1}<\abs{F_2}<\dots$.
The set $I=\{\abs{F_n}:n\in\Nset\}$ and the families
$\mc F_{\abs{F_n}}=\bigl\{B\bigl(x,\frac12\eps_{\abs{F_n}}\bigr):
x\in F_n\bigr\}$, $n\in\Nset$
obviously witness condition (ii) of Theorem~\ref{combDnull}.
\end{proof}

\section{\uhnull{}-sets vs.~$\MM$-additive and $\EE$-additive sets}
\label{sec:meager}

Recall that $\MM$ denotes the ideal of meager sets
in $\Cset$ and $\EE$ is the intersection ideal in $\Cset$.
Recall that a set $X\subs\Cset$ is termed \emph{strongly null} if
$X+M\neq\Cset$ for each meager set $M$.
The famous Galvin-Mycielski-Solovay Theorem asserts that a subset of $\Cset$ is
strongly null if and only if it is \smz{}.
Together with Theorem~\ref{hnullsmz} it yields:
\begin{thm}
A set $X\subs\Cset$ is \hnull{} if and only if it is strongly null. \label{mainSN}
\end{thm}

Recall that given an ideal $\mc J$ on $\Cset$, a set $X\subs\Cset$ is termed
\emph{$\mc J$-additive} if $X+J\in\mc J$ for each $J\in\mc J$.
Inspired by this theorem, we attempt to establish, for subsets of $\Cset$, similar connections
between \hnull{} and $\MM$- and $\EE$-additive sets (this section),
\upnull{} and $\NN$-additive sets (Section~\ref{sec:nadd}) and
\dpnull{} and \Tprime-sets (Section~\ref{sec:TT}), respectively.

Besides $\mc J$-additive sets we also define a stronger notion of sharply
$\mc J$-additive sets.
\begin{defn}
Given an ideal $\mc J$ on $\Cset$, a set $X\subs\Cset$
is termed \emph{sharply $\mc J$-additive} if for every $J\in \mc J$
there is a \si compact set $F\sups X$ such that $F+J\in\mc J$.

We also define a set $X\subs\Cset$ to be \emph{sharply null}
if for each $M\in\MM$ there is a \si compact set $F\sups X$ such that
$F+M\neq\Cset$.

It is clear that a sharply $\mc J$-additive set is $\mc J$-additive
and that a sharply null set is strongly null.
\end{defn}

In this section we prove the following theorem that in particular implies
that a set $X\subs\Cset$ is $\MM$-additive if and only if
it is $\EE$-additive.

\begin{thm}\label{mainME}
For any set $X\subs\Cset$, the following are equivalent.
\begin{enum}
\item $X$ is \uhnull{}, \label{uhnull}
\item $X$ is $\MM$-additive, \label{Madd}
\item $X$ is $\EE$-additive, \label{Eadd}
\item $X$ is sharply $\MM$-additive, \label{sMadd}
\item $X$ is sharply $\EE$-additive, \label{sharpE}
\item $X$ is sharply null. \label{sharpN}
\end{enum}
\end{thm}
\begin{proof}
We shall prove now  \eqref{uhnull}$\Implies$\eqref{Eadd} and
\eqref{sharpE}$\Implies$\eqref{sharpN}$\Implies$\eqref{sMadd}$\Implies$\eqref{Madd}.
The implications \eqref{Madd}$\Implies$\eqref{uhnull} and \eqref{Eadd}$\Implies$\eqref{sharpE}
are subject to standalone Propositions~\ref{MeShelah} and~\ref{EisEsharp}.

\smallskip
\eqref{uhnull}$\Implies$\eqref{Eadd}:
Assume $X$ be \uhnull. Let $E\in\EE$.
By Theorem~\ref{prodUhnull}, $X\times E\in\NNs(\uhm_0^1)$.
Since the mapping $(x,y)\mapsto x+y$ is Lipschitz, Lemma~\ref{lipschitz} yields
$X+E\in\NNs(\uhm_0^1)$.

\smallskip
\eqref{sharpE}$\Implies$\eqref{sharpN}:
We employ a Pawlikowski's~\cite{MR1380640} theorem, see also~\cite[Theorem 8.1.19]{MR1350295}:
\emph{For each $M\in\MM$ there exists $E\in\EE$ such that for each $Y\subs\Cset$,
if $Y+E\in\NN$, then $Y+M\neq\Cset$.}

Suppose $X$ is sharply $\EE$-additive. Let $M\in\MM$.
Let $E\in\EE$ be the set guaranteed by the Pawlikowski's theorem.
Since $X$ is sharply $\EE$-additive, there is $F\sups X$
\si compact such that $F+E\in\EE\subs\NN$.
Therefore $F+M\neq\Cset$.
Thus $X$ is sharply null.

\smallskip
\eqref{sharpN}$\Implies$\eqref{sMadd}:
Suppose $X$ is sharply null and let $M\in\MM$.
We may assume that $M$ is \si compact.
Let $Q\subs\Cset$ be a countable dense set.
Since $Q$ is countable, $Q+M$ is meager.
Therefore there is $F\sups X$ such that $Q+M+F\neq\Cset$.
Let $z\notin Q+M+F$. Then, for all $q\in Q$, $z\notin q+M+F$, i.e.~%
$z+q\notin M+F$. Therefore $(M+F)\cap(Q+z)=\emptyset$.
Since $Q$ is dense, so is $Q+z$.
Therefore the complement of $F+M$ is dense.

Since $F+M$ is a continuous image of a \si compact set $F\times M$,
it is \si compact as well.
Since it has a dense complement, it is meager by
the Baire category theorem.

\smallskip
\eqref{sMadd}$\Implies$\eqref{Madd} is obvious.
\end{proof}
In order to prove that every $\MM$-additive set is \uhnull{} we need
a Shelah's~\cite{MR1324470} characterization of $\MM$-additive sets:
\begin{lem}[{\cite[Theorem 2.7.17]{MR1350295}}]\label{ShelahM}
$X$ is $\MM$-additive if and only if
\begin{multline*}
  \forall f\in\UPset\,\,\exists g\in\Pset\,\,\exists y\in\Cset\,\,
  \forall x\in X\,\,\forall^\infty n\,\,\exists k\\
  g(n)\leq f(k)<f(k+1)\leq g(n+1) \&\
  x\rest [f(k),f(k+1))=y\rest [f(k),f(k+1)).
\end{multline*}
\end{lem}
\begin{prop}\label{MeShelah}
If $X\subs\Cset$ is $\MM$-additive, then $X$ is \uhnull.
\end{prop}
\begin{proof}
Let $X\subs\Cset$ be $\MM$-additive.
Let $h$ be a Hausdorff function.
We have to show that $\uhm^h(X)=0$.
Define recursively $f\in\UPset$
to satisfy
$$
  2^{f(k)}\cdot h\bigl(2^{-f(k+1)}\bigr)\leq 2^{-k},\quad k\in\Nset.
$$
By Lemma~\ref{ShelahM} there is $g\in\Pset$ and $y\in\Cset$ such that
\begin{multline}\label{ShelahM2}
  \forall x\in X\,\,\forall^\infty n\,\,\exists k\\
  g(n)\leq f(k)<g(n+1)\ \&\
  x\rest [f(k),f(k+1))=y\rest [f(k),f(k+1)).
\end{multline}
For $p\in\CCset$ denote $[p]=\{f\in\Cset:p\subs f\}$ the corresponding
cylinder. Define
\begin{alignat*}{3}
  &\mc B_k&&=
  \bigl\{
  \bigl[p\concat y\rest [f(k),f(k+1))\bigr]:p\in 2^{f(k)}
  \bigr\},\qquad
  && k\in\Nset,\\
  &\mc G_n&&=\bigcup
  \bigl\{\mc B_k:g(n)\leq f(k)<g(n+1)\bigr\},
  && n\in\Nset,\\
 & \mc B&&=\bigcup_{k\in\Nset}\mc B_k=\bigcup_{n\in\Nset}\mc G_n.
\end{alignat*}
With this notation~\eqref{ShelahM2} reads
\begin{equation}\label{ShelahM22}
  \forall x\in X\,\,\forall^\infty n\,\,\exists G\in\mc G_n\,\,
  x\in G.
\end{equation}
Since each of the families $\mc G_n$ is finite,
it follows that $\mc G_n$'s witness that $\mc B$ is a $\gamma$-groupable
cover of $X$.
Using Lemma~\ref{gammagr} it remains to show
that the Hausdorff sum $\sum_{B\in\mc B}h(dB)$ is finite.
Since $\abs{\mc B_k}=2^{f(k)}$ and
$dB=2^{-f(k+1)}$ for all $k$ and all $B\in\mc B_k$, we have
$$
  \sum_{B\in\mc B}h(dB)=
  \sum_{k\in\Nset}\sum_{B\in\mc B_k}h(dB)=
  \sum_{k\in\Nset}2^{f(k)}\cdot h(2^{-f(k+1)})
  \leq\sum_{k\in\Nset}2^{-k}<\infty.
  \qedhere
$$
\end{proof}

In order to prove that every $\EE$-additive set is sharply $\EE$-additive, we
employ a combinatorial description of closed null sets
given by Bartoszynski and Shelah~\cite{MR1186905},
see also~\cite[2.6.A]{MR1350295}.
For $f\in\UPset$ let
$$
  \CCC_f=\left\{\seq{F_n}:\forall n\in\Nset\left(F_n\subs2^{[f(n),f(n+1))}\ \&\
         \frac{\abs{F_n}}{2^{f(n+1)}}\leq\frac{1}{2^n}\right)\right\}
$$
and for $f\in\UPset$ and $F\in\CCC_f$ define
$$
  S(f,F)=\{z\in\Cset:\fmany n\in\Nset\ z\rest[f(n),f(n+1))\in F_n\}.
$$
It is easy to check that $S(f,F)\in\EE$ for all $f\in\UPset$ and $F\in\CCC_f$.
By \cite[Theorem 4.2]{MR1186905} (or~\cite[2.6.3]{MR1350295}),
these sets actually form a base of $\EE$. The proof therein
yields a little more:
\begin{lem}\label{2.6.3}
For each $E\in\EE$ and each $f\in\UPset$ there is $g\in\UPset$ and $G\in\CCC_{f{\circ}g}$
such that $E\subs S(f{\circ}g,G)$.
\end{lem}
\begin{lem}\label{Einc3}
Let $f,g\in\UPset$, $F\in\CCC_f$ and $G\in\CCC_{f{\circ}g}$. Then
$S(f,F)\subs S(f{\circ}g,G)$ if and only if
\begin{equation}\label{Einc}
  \fmany n\in\Nset\ \forall k\in[g(n),g(n+1))\quad F_k\subs\{z\rest[f(k),f(k+1)):z\in G_n\}.
\end{equation}
\end{lem}
\begin{proof}
Suppose condition~\eqref{Einc} fails. Then there is $I\in[\Nset]^\Nset$ such that
\begin{equation}\label{Einc2}
\forall n\in I\ \exists k_n\in[g(n),g(n+1))\ \exists z_{k_n}\in F_{k_n}\
  \forall z\in G_n\ z_{k_n}\nsubseteq z.
\end{equation}
For each $k\notin\{k_n:n\in I\}$ choose $z_k\in F_k$ and let $z\in\Cset$ be a sequence that extends
simultaneously all $z_k$'s. Then obviously $z\in S(f,F)$.
On the other hand, condition~\eqref{Einc2} ensures that
$z\notin S(f{\circ}g,G)$.
Thus $S(f,F)\subs S(f{\circ}g,G)$ yields \eqref{Einc}.
The reverse implication is straightforward.
\end{proof}

\begin{prop}\label{EisEsharp}
If $X\subs\Cset$ is $\EE$-additive, then $X$ is sharply $\EE$-additive.
\end{prop}
\begin{proof}
Suppose $X$ is $\EE$-additive. Let $E\in\EE$.
We are looking for an $F_\sigma$-set $\widetilde X\sups X$ such that
$\widetilde X+E\in\EE$.

There are $f\in\UPset$ and $F\in\CCC_f$ such that
$E\subs S(f,F)$. Since $S(f,F)\in\EE$, we have $X+S(f,F)\in\EE$. By Lemma~\ref{2.6.3} there are
$g$ and $G\in\CCC_{f{\circ}g}$
such that $X+S(f,F)\subs S(f{\circ}g,G)$, i.e.
$x+S(f,F)\subs S(f{\circ}g,G)$ for all $x\in X$.

The set $\widetilde X$ we are looking for is
$$
  \widetilde X=\{x\in\Cset:x+S(f,F)\subs S(f{\circ}g,G)\}.
$$
Obviously $X\subs\widetilde X$. It is also obvious that
$\widetilde X+E\subs\widetilde X+S(f,F)\subs S(f{\circ}g,G)\in\EE$.
Thus it remains to show that $\widetilde X$ is $F_\sigma$.

For any $x\in\Cset$ and $k\in\Nset$ set
$$
  F_k^x=\{z+x\rest[f(k),f(k+1)):z\in F_k\}
$$
and consider the sequence $F^x=\seq{F_k^x}$.
Clearly $F^x\in\CCC_f$ and $S(f,F^x)=x+S(f,F)$.
Therefore $\widetilde X=\{x\in\Cset:S(f,F^x)\subs S(f{\circ}g,G)\}$.
Use Lemma~\ref{Einc3} to conclude that
$$
  x\in\widetilde X\Leftrightarrow
  \fmany n\in\Nset\ \forall k\in[g(n),g(n+1))\ F_k^x\subs\{z\rest[f(k),f(k+1)):z\in G_n\}.
$$
It follows that $\widetilde X$ is $F_\sigma$ as long as the sets
$$
  A_{n,k}=\{x\in\Cset:F_k^x\subs\{z\rest[f(k),f(k+1)):z\in G_n\}\}
$$
are closed.
Fix $n\in\Nset$ and $k\in[g(n),g(n+1))$. Unraveling the
definitions yields
$$
  x\in A_{n,k}\Leftrightarrow
  \exists y\in2^{[f(k),f(k+1))}\quad y\subs x\ \&\
  \forall z\in F_k\ \exists t\in G_n\
  z+y\subs t.
$$
Since all three sets $2^{[f(k),f(k+1))}$, $F_k$ and $G_n$ are finite, the set $A_{n,k}$ is even clopen.
We are done.
\end{proof}
\noindent
The proof of Theorem~\ref{mainME} is now complete.
\begin{coro}
Let $f:\Cset\to\Cset$ be a continuous mapping. If $X$ is $\MM$-additive, then so is $f(X)$.
\end{coro}
\begin{coro}
If $X\subs\Cset$ is $\MM$-additive, then $\phi(X\times E)\in\EE$
for each $E\in\EE$ and every Lipschitz mapping $\phi:\Cset\times\Cset\to\Cset$.
\end{coro}
Recall that \emph{transitive additivity} of an ideal $\mc J$ is defined by
$$
  \add^\star\mc J=\min\{\abs{X}:\exists J\in\mc J\ X+J\notin\mc J\}.
$$
Transitive additivity and other transitive coefficients are
investigated in detail in \cite[2.7]{MR1350295}. The following is an obvious consequence
of the equivalence $\MM$-additive $\Leftrightarrow$ $\EE$-additive.
\begin{coro}
$\add^\star\EE=\add^\star\MM$.
\end{coro}
Recall that a set $X\subs\Cset$ is \emph{transitively meager}
(or \emph{meager in the transitive sense}, or an \emph{$\mc{AFC}'$-set})
if for every perfect set $P\subs\Cset$ there is an
$F_\sigma$-set $F\sups X$ such that $(F+t)\cap P$ is meager in $P$ for all $t\in\Cset$.
These sets are investigated e.g.~in~\cite{MR1905154}.
One can prove that if $X$ is $\MM$-additive and $Y$ is transitively meager, then
$X+Y$ is transitively meager, i.e.~that $\MM$-additive sets are $\mc{AFC}'$-additive, but
that requires a nontrivial proof.
Nowik, Scheepers and Weiss~\cite[Theorem 9]{MR1610427}
have that every strongly meager set $X\subs\Cset$
(i.e.~$X+N\neq\Cset$ for all $N\in\NN$) is transitively meager.
The following statement follows at once from Lemma~\ref{lemUM}.
\begin{coro}
Every $\MM$-additive set is universally meager and transitively meager.
\end{coro}
\begin{proof}
Let $E\subs\Cset$ be $\MM$-additive, i.e.~\uhnull{}. It is universally meager by
Theorem~\ref{umg}.
To show it is transitively meager, let $P\subs\Cset$
be a perfect set and apply Lemma~\ref{lemUM}
with $Z=X=\Cset$, $Y=P$, $\phi=\id_P$ and $\mc F=\{x\mapsto x+t:t\in\Cset\}$.
\end{proof}

\section{\upnull{} sets vs.~$\NN$-additive sets}\label{sec:nadd}

The following theorem in particular shows that a set in $\Cset$ is \upnull{}
if and only if it is $\NN$-additive.

\begin{thm}\label{mainNN}
For any set $X\subs\Cset$, the following are equivalent.
\begin{enum}
\item $X$ is \upnull,
\item $X$ is $\NN$-additive,
\item $X$ is sharply $\NN$-additive,
\item $\hm^1(X\times N)=0$ for each $N\in\NN$.
\end{enum}
\end{thm}
We employ Shelah's~\cite{MR1324470} characterization of $\NN$-additive sets.
\begin{lem}[{\cite[Theorem 2.7.18]{MR1350295}}]\label{ShelahN}
$X\subs\Cset$ is $\NN$-additive if and only if
for each $f\in\UPset$
there is a sequence $\seq{H_n:n\in \Nset}$ such that
\begin{enum}
\item $\forall n\ H_n\subs 2^{[f(n),f(n+1))}$,
\item $\forall n\ \abs{H_n}\leq n$,
\item $X\subs\{x\in\Cset:\forall^\infty n\
       x\rest [f(n),f(n+1))\in H_n\}$.
\end{enum}
\end{lem}
\begin{proof}[Proof of Theorem~\ref{mainNN}]
(i)\Implies(iii):
Suppose $X$ is \upnull{} and $N\in\NN$. By Lemma~\ref{lemHaus}(ii) there is $h\prec 1$
such that $\hm^h(N)=0$. Let $g\in\HH$ be such that $gh\geq1$.
Then Theorem~\ref{basicPnull}(iv) and Lemma~\ref{lemP}(iii) yield a \si compact set $F$ such that
$\ubox^g(F)=0$. Apply Lemma~\ref{howroyd}(ii) to get
$$
  \hm^1(F\times N)\leq\hm^{gh}(F\times N)\leq\ubox^g(F)\hm^h(N)=0.
$$
Since $(x,y)\mapsto x+y$ is clearly a Lipschitz mapping, Lemma~\ref{lipschitz} yields
$\hm^1(F+N)=0$, i.e.~$F+N\in\NN$, as required.

This argument also proves (iv)\Implies(ii). (i)\Implies(iv) is nothing but Lemma~\ref{XN}(ii) and
(iii)\Implies(ii) is trivial.

(ii)\Implies(i):
Suppose that $X\subs\Cset$ is $\NN$-additive.
Let $h\in\HH$. We verify that $\ubox^h(X)\leq 1$.
Choose $F\in\UPset$ to satisfy $F(n)\leq\frac{1}{h(2^{1-n})}$.
Define recursively $f\in\UPset$ subject to
$$
  2^{f(n)}\, (n+1)!\leq
  F(f(n+1)),\quad n\in\Nset.
$$
Obviously
\begin{equation}\label{fuj}
  \forall n\,\, \forall k>n\quad
  2^{f(n)}\, k!\leq
  F(f(k)).
\end{equation}
Let $\seq{H_n}$ be the sequence guaranteed by the lemma.
Set
$$
  X_n=\{x\in\Cset:\forall k\geq n\ x\rest[f(k),f(k+1))\in H_n\},\quad n\in\Nset.
$$
We verify that $N_{X_n}(2^{-i})\leq F(i)$
for each $n$ and all $i\geq f(n+1)$. Let $k$ be the unique integer satisfying
$f(k)\leq i<f(k+1)$. In particular, $k>n$. It is obvious that
$$
  N_{X_n}(2^{-i})\leq 2^{f(n)}
  \abs{H_n}\cdot\abs{H_{n+1}}\cdot\dots\cdot\abs{H_k}\leq
  2^{f(n)}n(n+1)\dots k.
$$
Therefore~\eqref{fuj} yields $N_{X_n}(2^{-i})\leq F(f(k))\leq F(i)$.
The definition of $F$ thus yields $N_{X_n}(2^{-i})h(2^{1-i})\leq 1$ and
therefore
$$
  \ubox_0^h(X_k)=\varlimsup_{r\to0}N_{X_k}(r)\cdot h(r)
  \leq\varlimsup_{i\to\infty}N_{X_k}(2^{-i})\cdot h(2^{1-i})
  \leq 1.
$$
Condition (iii) of Lemma~\ref{ShelahN} ensures that $X_n\upto X$.
Therefore $\ubox^h(X)\leq1$ by Lemma~\ref{IncreasingSetsLemma}.
\end{proof}

\begin{coro}
Let $X\subs\Cset$ and $f:\Cset\to\Cset$ a continuous mapping.
If $X$ is $\NN$-additive, then so is $f(X)$.
\end{coro}

\section{\Dpnull{} sets vs.~\Tprime-sets}
\label{sec:TT}
Inspired by Shelah's theorem (cf.~Lemma~\ref{ShelahN}) Nowik and Weiss~\cite{MR1905154}
introduced and investigated the following notion.

\begin{defn}[{\cite{MR1905154}}]\label{defT'}
$X$ is called a \Tprime{}\emph{-set} if there exists $g\in\Pset$
such that for each $f\in\UPset$ there is
$I\in[\Nset]^\Nset$ and a sequence $\seq{H_n:n\in I}$ such that
\begin{enum}
\item $\forall n\in I\ H_n\subs 2^{[f(n),f(n+1))}$,
\item $\forall n\in I\ \abs{H_n}\leq g(n)$,
\item $X\subs\{x\in\Cset:\forall^\infty n\in I\
       x\rest [f(n),f(n+1))\in H_n\}$.
\end{enum}
\end{defn}
They proved a number of results on \Tprime{}-sets, e.g.~that \Tprime{}-sets
are Ramsey null. By proving that every $\gamma$-set is \Tprime{}
they showed that $\NN$-additive sets are consistently a proper subclass of \Tprime{}-sets.
They also proved that every \Tprime{}-set is $\MM$-additive and provided a CH example
of an $\MM$-additive set that is not \Tprime{}.

Nowik and Weiss ask at the end of~\cite{MR1905154} if \Tprime{}-sets coincide with
$\EE$-additive sets. In view of Theorem~\ref{mainME}, their example proves that it is
not so: Every \Tprime{}-set is $\EE$-additive, but under CH the converse fails.
To date it is not known if there is some natural ideal $\mc J$ such that \Tprime{}-sets
coincide with $\mc J$-additive sets.

In this section we prove that \Tprime{}-sets coincide with \dpnull{} sets.

\begin{thm}
Let $X\subs\Cset$. The following are equivalent.
\begin{enum}
\item $X$ is \dpnull,
\item $X$ is a \Tprime{}-set.
\end{enum}
\label{mainTT}
\end{thm}

\begin{proof}
(i)\Implies(ii):
Let $f\in\UPset$ and set $\eps_n=2^{-f(n+1)}$.
Let $I$ and $\mc F_n$'s be as in Theorem~\ref{combDnull}(ii).
By the choice of $\eps_n$ we may assume that each set $F\in\mc F_n$ is a cylinder
generated by some $p_F\in2^{f(n+1)}$.
For $n\in I$ set
$$
  H_n=\{p_F\rest[f(n),f(n+1)):F\in\mc F_n\}.
$$
Clearly $\abs{H_n}\leq\abs{\mc F_n}\leq n$.
Condition (iii) of Definition~\ref{mainTT} follows from the fact that
$\{\bigcup\mc F_n:n\in I\}$ is a $\gamma$-cover of $X$.

(ii)\Implies(i):
Let $h\in\HH$. We verify that $\dbox^h(X)\leq 1$.
Choose $G\in\UPset$ to satisfy $G(n)\leq\frac{1}{h(2^{-n})}$.
Let $g\in\Pset$ be the function from the Definition~\ref{defT'} of \Tprime{}.
Define recursively $f\in\UPset$ to satisfy
$$
  2^{f(n)}\cdot g(n)\leq G(f(n+1)).
$$
Let $I\in[\Nset]^\Nset$ and $\seq{H_n:n\in I}$
be as in the definition of \Tprime{} . Set
\begin{align*}
  F_n&=\{x\in\Cset:x\rest[f(n),f(n+1))\in H_n\},\quad n\in I,\\
  X_k&= \bigcap_{n\geq k,\, n\in I} F_n,\quad k\in\Nset
\end{align*}
and for each $n\in I$ put $\eps_n=2^{-f(n+1)}$.
Fix $k$. It is obvious that
$$
  N_{X_k}(\eps_n)\leq N_{F_n}(\eps_n)\leq
  2^{f(n)}\cdot\abs{H_n}\leq
  2^{f(n)}\cdot g(n)\leq G(f(n+1))
$$
holds for each $n\geq k,\,n\in I$ and since
$G(f(n+1))\leq\frac{1}{h(2^{-f(n+1)})}=\frac{1}{h(\eps_n)}$,
we finally get
$$
  \lbox_0^h(X_k)\leq\varliminf_{n\in I}N_{X_k}(\eps_n)\cdot h(\eps_n)
  \leq 1.
$$
Condition~\ref{defT'}(iii) guarantees that
$X_k\upto X$. Hence $\dbox^h(X)\leq 1$,
as required.
\end{proof}
\begin{coro}
Let $X\subs\Cset$ and $f:\Cset\to\Cset$ a continuous mapping.
If $X$ is \Tprime{}, then so is $f(X)$.
\end{coro}

\bibliographystyle{amsplain}

\begin{thebibliography}{10}

\bibitem{MR1350295}
Tomek Bartoszy{\'n}ski and Haim Judah, \emph{Set theory}, A K Peters Ltd.,
  Wellesley, MA, 1995, On the structure of the real line. \MR{1350295
  (96k:03002)}

\bibitem{MR1186905}
Tomek Bartoszy{\'n}ski and Saharon Shelah, \emph{Closed measure zero sets},
  Ann. Pure Appl. Logic \textbf{58} (1992), no.~2, 93--110. \MR{1186905
  (94b:03084)}

\bibitem{MR1555386}
A.~S. Besicovitch, \emph{Concentrated and rarified sets of points}, Acta Math.
  \textbf{62} (1933), no.~1, 289--300. \MR{1555386}

\bibitem{MR1555389}
\bysame, \emph{Correction}, Acta Math. \textbf{62} (1933), no.~1, 317--318.
  \MR{1555389}

\bibitem{MR738943}
Fred Galvin and Arnold~W. Miller, \emph{{$\gamma $}-sets and other singular
  sets of real numbers}, Topology Appl. \textbf{17} (1984), no.~2, 145--155.
  \MR{738943 (85f:54011)}

\bibitem{howroydPhD}
J.~D. Howroyd, \emph{On the theory of {H}ausdorff measures in metric spaces},
  Ph.D. thesis, University College, London, 1994.

\bibitem{MR1362951}
\bysame, \emph{On {H}ausdorff and packing dimension of product spaces}, Math.
  Proc. Cambridge Philos. Soc. \textbf{119} (1996), no.~4, 715--727.
  \MR{1362951 (96j:28006)}

\bibitem{Kelly}
J.~D. Kelly, \emph{{A Method for Constructing Measures Appropriate for the
  Study of Cartesian Products}}, Proc. London Math. Soc. \textbf{s3-26} (1973),
  no.~3, 521--546.

\bibitem{MR0422027}
Richard Laver, \emph{On the consistency of {B}orel's conjecture}, Acta Math.
  \textbf{137} (1976), no.~3-4, 151--169. \MR{0422027 (54 \#10019)}

\bibitem{MR1610427}
Andrej Nowik, Marion Scheepers, and Tomasz Weiss, \emph{The algebraic sum of
  sets of real numbers with strong measure zero sets}, J. Symbolic Logic
  \textbf{63} (1998), no.~1, 301--324. \MR{1610427 (99c:54049)}

\bibitem{MR1905154}
Andrzej Nowik and Tomasz Weiss, \emph{On the {R}amseyan properties of some
  special subsets of {$2\sp \omega$} and their algebraic sums}, J. Symbolic
  Logic \textbf{67} (2002), no.~2, 547--556. \MR{1905154 (2003c:03087)}

\bibitem{MR1380640}
Janusz Pawlikowski, \emph{A characterization of strong measure zero sets},
  Israel J. Math. \textbf{93} (1996), 171--183. \MR{1380640 (97f:28003)}

\bibitem{MR0281862}
C.~A. Rogers, \emph{Hausdorff measures}, Cambridge University Press, London,
  1970. \MR{0281862 (43 \#7576)}

\bibitem{MR1779763}
Marion Scheepers, \emph{Finite powers of strong measure zero sets}, J. Symbolic
  Logic \textbf{64} (1999), no.~3, 1295--1306. \MR{1779763 (2001m:03096)}

\bibitem{MR1324470}
Saharon Shelah, \emph{Every null-additive set is meager-additive}, Israel J.
  Math. \textbf{89} (1995), no.~1-3, 357--376. \MR{1324470 (96j:04002)}

\bibitem{MR2177439}
Boaz Tsaban and Tomasz Weiss, \emph{Products of special sets of real numbers},
  Real Anal. Exchange \textbf{30} (2004/05), no.~2, 819--835. \MR{2177439
  (2006g:26005)}

\bibitem{MR1814112}
Piotr Zakrzewski, \emph{Universally meager sets}, Proc. Amer. Math. Soc.
  \textbf{129} (2001), no.~6, 1793--1798 (electronic). \MR{1814112
  (2001m:03097)}

\bibitem{MR2427418}
\bysame, \emph{Universally meager sets. {II}}, Topology Appl. \textbf{155}
  (2008), no.~13, 1445--1449. \MR{2427418 (2010d:03080)}

\bibitem{ZinPack}
Ond{\v{r}}ej Zindulka, \emph{Packing measures and dimensions on {C}artesian
  products}, Publ. Math., to appear.

\end{thebibliography}
\providecommand{\bysame}{\leavevmode\hbox to3em{\hrulefill}\thinspace}
\providecommand{\MR}{\relax\ifhmode\unskip\space\fi MR }
\providecommand{\MRhref}[2]{%
  \href{http://www.ams.org/mathscinet-getitem?mr=#1}{#2}
}
\providecommand{\href}[2]{#2}

\end{document}